\documentclass[oneside]{amsart}

\usepackage{amssymb}
\usepackage{graphicx}
\usepackage[latin1]{inputenc}
\usepackage{amssymb,amsmath}
\usepackage{verbatim}
\usepackage{color}
\usepackage{enumerate}
\usepackage{url}
\usepackage{geometry}
\usepackage{hyperref}

\newtheorem{theorem}{Theorem}[section]
\newtheorem{lemma}[theorem]{Lemma}
\newtheorem{corollary}[theorem]{Corollary}
\newtheorem{proposition}[theorem]{Proposition}

\theoremstyle{definition}

\newtheorem{example}[theorem]{Example}

\newtheorem{ass}[theorem]{Assumption}

\theoremstyle{remark}
\newtheorem{remark}[theorem]{Remark}

\numberwithin{equation}{section}

\newcommand\Z{\ensuremath{\mathbb Z}}\newcommand\A{\ensuremath{\mathbb A}}
\newcommand\Q{\ensuremath{\mathbb Q}}\newcommand\R{\ensuremath{\mathbb R}}
\newcommand\C{\ensuremath{\mathbb C}}

\newcommand\Aut{\operatorname{Aut}}

\newcommand\Frob{\operatorname{Frob}}

\newcommand\GL{\operatorname{GL}}

\newcommand\Ind{\operatorname{Ind}}

\newcommand\Sym{\operatorname{Sym}}

\newcommand\Tr{\operatorname{Tr}}

\newcommand{\fp}{{\mathfrak{p}}}

\newcommand{\fN}{{\mathfrak{N}}}

\newcommand{\USp}{\mathrm{USp}}
\newcommand{\U}{\mathrm{U}}

\def\p{\mathfrak p}

\newcommand{\comp}{\begin{picture}(6,5)(-3,-2)\put(0,1){\circle{2}} \end{picture}}\def\circ{\comp}
\def\fN{\mathfrak N}

\newcommand{\ra}{\rightarrow}

\def\XXint#1#2#3{{\setbox0=\hbox{$#1{#2#3}{\int}$}
\vcenter{\hbox{$#2#3$}}\kern-.5\wd0}}

\begin{document}

\title{On the rank and the convergence rate towards the Sato--Tate measure}
\keywords{}

\author{Francesc Fit\'e}
\address{Departament de Matem\`atiques, Universitat Polit\`ecnica de Catalunya/BGSmath \\Edifici Omega, C/Jordi Girona 1--3\\
08034 Barcelona, Catalonia
}
\email{francesc.fite@gmail.com}
\urladdr{https://mat-web.upc.edu/people/francesc.fite/}

\author{Xavier Guitart}
\address{Departament de Matemàtiques i Informàtica\\
Universitat de Barcelona\\Gran via de les Corts Catalanes, 585\\
08007 Barcelona, Catalonia
}
\email{xevi.guitart@gmail.com}

\thanks{Fit\'e was funded by the German Research Council via SFB/TR 45 and by the Excellence Program Mar\'ia de Maeztu MDM-2014-0445. Guitart was partially supported by MTM2015-66716-P.
Both authors were partially supported by MTM2015-63829-P. This project has received funding from the European Research Council (ERC) under the European Union's Horizon 2020 research and innovation programme (grant agreement No 682152).}

\subjclass[2010]{11G05, 11G10, 14G10, 11M50}

\date{\today}

\dedicatory{}

\begin{abstract}
Let $A$ be an abelian variety defined over a number field and let $G$ denote its Sato--Tate group. Under the assumption of certain standard conjectures on $L$-functions attached to the irreducible representations of $G$, we study the convergence rate of any virtual character of~$G$. We find that this convergence rate is dictated by several arithmetic invariants of $A$, such as its rank or its Sato--Tate group $G$. The results are consonant with some previous experimental observations, and we also provide additional numerical evidence consistent with them.
The techniques that we use were introduced by Sarnak in a letter to Mazur, in order to explain the bias in the sign of the Frobenius traces of an elliptic curve without complex multiplication defined over~$\Q$. We show that the same methods can be adapted to study the convergence rate of the characters of its Sato--Tate group, and that they can also be employed in the more general case of abelian varieties over number fields. A key tool in our analysis is the existence of limiting distributions for automorphic $L$-functions, which is due to Akbary, Ng, and Shahabi.
\end{abstract}

\maketitle
\tableofcontents

\section{Introduction}\label{section: intro}

Let $k$ be a number field and $A$ be an abelian variety defined over $k$ of dimension $g\geq 1$. Following Serre \cite{Ser94}, Banaszak and Kedlaya \cite{BK16a} have attached to $A$ a compact real Lie subgroup $G$ of $\USp(2g)$, the so-called Sato--Tate group of $A$, with the conjectural property that it governs the distribution of the Frobenius elements attached to $A$. \footnote{The construction of \cite{BK16a} is in fact more general and it applies to odd weight motives.}

In order to make a more precise statement, let us introduce some notations. Let $\ell$ be a rational prime and let $V_\ell(A)$ denote the (rational) $\ell$-adic Tate module of $A$. The action of the absolute Galois group $G_k$ of $k$ on $V_\ell(A)$ gives rise to an $\ell$-adic representation
\begin{equation}\label{equation: ladic rep}
\varrho_A:G_k\ra \Aut(V_\ell(A))\,.
\end{equation} 
Denote by $P$ the set of nonzero prime ideals of $k$ lot lying over $\ell$ and of good reduction for $A$, that is, the set of nonzero prime ideals of $k$ not dividing the conductor $\fN$ of $A$. For a prime $\fp$ in $P$, set
\begin{equation}\label{equation: Lpoly}
L_\fp(A,T):=\det(1-\varrho_A(\Frob_\fp)T\,|\,V_\ell(A))\,,
\end{equation} 
where $\Frob_\fp$ denotes a Frobenius element at $\fp$. Attached to $\fp$, one can construct a semisimple conjugacy class $y_\fp$ in the set of conjugacy classes $Y$ of $G$ such that 
\begin{equation}\label{equation: yp}
\det(1-y_{\fp}T) = L_\fp(A,T/\sqrt{|\fp|})\,,
\end{equation} 
where $|\fp|$ denotes the absolute norm of $\fp$. We will refer to the projection $\mu$ of the Haar measure $\mu_G$ of $G$ on $Y$ as the \emph{Sato--Tate measure} of $A$. The conjectural property of $G$ that we have alluded to before predicts that the sequence $\{y_{\fp}\}_{\fp \in P}$, where the ideals in~$P$ are ordered according to their absolute norm, is equidistributed on~$Y$ with respect to~$\mu$. 

Let us recall what this means. For $x>0$, let $\pi(x)$ denote the number of primes $\fp$ in $P$ such that $|\fp|\leq x$. Set $$
\mu_x:=\frac{1}{\pi(x)}\sum_{|\fp|\leq x} \delta_{y_{\fp}}\,,
$$
where $\delta_{y_\fp}$ denotes the Dirac measure at $y_\fp$ and the sum runs over primes of~$P$ such that $|\fp|\leq x$. From now on, we make the convention that all sums of terms involving $y_\p$ run over primes $\fp\in P$. By definition, we say that $\{y_{\fp}\}_{\fp \in P}$ is equidistributed on~$Y$ with respect to~$\mu$, or simply $\mu$-equidistributed on~$Y$, if 
\begin{equation}\label{equation: weakly}
\mu_x\ra \mu\qquad\text{ weakly as }\qquad x\ra \infty\,.
\end{equation}
As explained in \cite[Prop. 2, App. Chap. I]{Ser68}, the sequence $\{y_\fp\}_{\fp\in P}$ is $\mu$-equidistributed on $Y$ if and only if for every irreducible character $\chi$ of $G$ one has that
\begin{equation}\label{equation: char of equid}
\lim_{x\ra \infty}\frac{1}{\pi(x)}\sum_{|\fp|\leq x}\chi(y_\fp)=\delta(\chi)\,,
\end{equation}
where $\delta(\chi)$ is $1$ or $0$ depending on whether $\chi$ is trivial or not.\footnote{More generally and for future use, given a virtual character $\varphi$ of $G$, let $\delta (\varphi)$ denote the multiplicity of the trivial representation in $\varphi$.} Recall that, by the Prime Number Theorem, \eqref{equation: char of equid} is equivalent to 
\begin{equation}\label{equation: equidPNT}
\lim_{x\ra \infty}\frac{\log(x)}{x}\sum_{|\fp|\leq x}\chi(y_\fp)=\delta(\chi)\,,
\end{equation}
and that, by the Abel summation trick, \eqref{equation: char of equid} is also equivalent to
\begin{equation}\label{equation: equidtrick}
\sum_{|\fp|\leq x}\chi(y_\fp)\log(|\fp|)= \delta(\chi)x + o(x)\,.
\end{equation}
 
It is of crucial importance that \eqref{equation: char of equid} is connected, via the Wiener--Ikehara Theorem, to the theory of $L$-functions (see \cite[Thm. 1, App. Chap. I]{Ser68}). For $\p\in P$, define the polynomial 
$$
L_\fp(\chi,T):= \det(1 - \varrho(y_{\fp})T)\,,
$$
where~$\varrho$ is an irreducible representation of~$G$ of character~$\chi$. The degree $d_\chi(\fp)$ of $L_\fp(\chi,T)$ is the degree $d_\chi$ of the representation~$\varrho$, and the roots of this polynomial all have absolute value $1$.
One finds that the sequence $\{y_\fp\}_{\fp\in P}$ is $\mu$-equidistributed on $Y$ if and only if, for every nontrivial irreducible character $\chi$ of $G$, the partial Euler product
\begin{equation}\label{equation: eulprod}
  L^P(\chi,s):= \prod_{\fp\in P} L_\fp(\chi,|\fp|^{-s})^{-1}\,,\qquad \text{defined for } \Re(s)>1\,, 
\end{equation}
extends to a holomorphic function on (an open neighborhood of) the halfplane $\Re(s)\geq 1$ and does not vanish at $s=1$. This is unknown in general, but it would follow from the automorphy of $L^P(\chi,s)$, which is predicted by the global Langlands correspondences. Throughout the paper, we will assume the automorphy of $L^P(\chi,s)$, together with a number of conjectural properties that $L^P(\chi,s)$ is expected to satisfy on the halfplane $\Re(s)\geq 1/2$. More precisely, we will consider the following assumption.

\begin{ass}\label{assumption: main} For every irreducible nontrivial representation $\varrho$ of $G$ of character $\chi$:
\begin{enumerate}
\item   The $L$-function $L^P(\chi,s)$ is automorphic. By this we mean that, for each $\fp|\fN$, there exist polynomials $L_\fp(\chi,T)$ of degree $d_\chi(\fp)\leq d_\chi$ such that the Euler product
\begin{equation}\label{equation: compL}
L(\chi,s):=L^P(\chi,s)\prod_{\p|\fN} L_\fp(\chi,|\fp|^{-s})^{-1}\,,\qquad \text{defined for } \Re(s)>1\,, 
\end{equation}
coincides with the automorphic $L$-function $L(\pi,s)$ of some irreducible unitary algebraic cuspidal representation $\pi$ of $\GL_{d_\chi[k:\Q]}(\A_{\Q})$. Thus, the function $L(\chi,s)$ extends to an analytic function on $\C$. 
\item The Riemann Hypothesis holds for $L^P(\chi,s)$ (equiv. for $L(\chi,s)$); that is, all the zeros $\sigma +i\gamma $ of $L^P(\chi,s)$ (equiv. of $L(\chi,s)$) in the critical region $0< \sigma < 1$ are in fact on the critical line $\sigma = 1/2$.

\end{enumerate}
\end{ass}
\begin{remark}
Note that (1) is implied by standard conjectures on automorphic representations. Indeed, the global Langlands correspondence implies that $L(\chi,s)$, as the $L$-function of an irreducible representation of the motivic Galois group of $k$, is the $L$-function of an irreducible unitary cuspidal algebraic representation of $\GL_{d_\chi}(\A_k)$ (see for example \cite[\S4.2]{Cog03a}). By automorphic induction (a consequence of the Principle of Functoriality, see for example \cite[\S4.1]{Cog03a}), $L(\chi,s)$ is then expected to be the $L$-function of an irreducible unitary algebraic  cuspidal representation of $\GL_{d_\chi[k:\Q]}(\A_{\Q})$. 
\end{remark}

This article is concerned with the study of the convergence rate of the measures $\mu_x$ towards the Sato-Tate measure $\mu$. There are several proposals in the literature to estimate this convergence rate. For example, Mazur \cite[\S 1.7,\S1.9]{mazur07} considers $L^\infty$ and $L^2$ norms of the error term of certain counting functions on intervals. In the present article, we adopt the related point of view of using counting functions attached to virtual characters of $G$.

For a (complex) virtual character $\varphi$ of $G$, that is, $\varphi \in \bigoplus_{\chi} \C\cdot \chi$, where $\chi$ runs over the irreducible characters of $G$, set 
\begin{equation}\label{equation: integral}
\delta(\varphi,x):=\frac{1}{\pi(x)} \sum_{|\fp|\leq x}\varphi(y_{\fp})\,. 
\end{equation} 
It follows from Assumption~\ref{assumption: main} (1), that for every nontrivial irreducible character $\chi$ of $G$ one has that
\begin{equation}\label{equation: chars}
\lim_{x\ra \infty} \delta(\chi,x)=0\,.
\end{equation}

It is then apparent that a way to estimate the rate of convergence of the measures $\mu_x$ towards the measure $\mu$ is by studying how fast the function $\delta(\chi,x)$ approaches the function $0$ as $x$ tends to $\infty$. There are examples of this in the literature: By generalizing \cite[Prop. 4.1]{Mur85} and under Assumption~\ref{assumption: main} (2), in (2.4) of \cite{BK16} one finds that
\begin{align}\label{eq:BK}
 \delta(\chi,x)=O\left(d_\chi[k:\Q]x^{-1/2}\log(N(x+d_\chi[k:\Q]))\log(x)\right)\,,
\end{align}
where $N:=|\fN|$. One may interpret the $O$-notation as a sort of \emph{asymptotic $L^\infty$-norm} (a supremum norm in a neighborhood of infinity) of the function $\delta(\chi,x)$. With this notion of convergence rate, Formula~\eqref{eq:BK} makes apparent how the rate of convergence depends on the conductor of~$A$.

The goal of this note is to study the influence on the convergence rate of other invariants of $A$, most notably (although not only) of the rank of $A$. For this purpose, we instead propose to use what could be seen as a sort of \emph{asymptotic $L^2$-norm}. For $X>0$, and $\varphi$ a virtual character of $G$ not containing the trivial character, define
$$
I(\varphi,X):=\frac{1}{\log(X)}\int_2^X |\delta(\varphi,x)|^2 dx\,. 
$$
The main goal of the paper is to study the asymptotic behavior of $I(\varphi,X)$ as $X\ra \infty$. The following is the main result.

\begin{theorem}\label{theorem: main}
Under Assumption~\ref{assumption: main}, for every virtual  character of the form $\varphi=\sum_{\chi\not =1} c_\chi \chi$, where $c_\chi\in \C$ and $\chi$ runs over the irreducible nontrivial characters of $G$, one has that
\begin{equation}\label{equation: main}
I(\varphi):=\lim_{X\ra \infty}I(\varphi,X)=\big|\sum_{\chi}c_\chi(2r_\chi+u_\chi)\big|^2+\sum_\chi\sum_{\gamma_\chi\neq 0}\frac{|c_\chi|^2}{1/4+\gamma_\chi^2}\,,
\end{equation}
where $r_\chi$ denotes the order of the zero of $L(\chi,s)$ at $s=1/2$, $u_\chi$ is the Frobenius--Schur index of $\chi$, and $\gamma_\chi$ runs over the non-zero imaginary parts of the zeros of $L(\chi,s)$ on the critical line.
\end{theorem}

Observe that via this theorem, assuming the Birch and Swinnerton--Dyer conjecture and taking for $\chi$ the character of the tautological representation of $G$ (seen as a subgroup of $\USp(2g)$), the influence of the rank of the Mordell-Weil group of $A$ on the rate of convergence of the measures~$\mu_x$ towards the Sato--Tate measure~$\mu$ becomes apparent.\footnote{An effect of the rank on the convergence towards the Sato--Tate measure had been experimentally observed. On Drew Sutherland's web page:
\begin{center}\url{https://math.mit.edu/~drew/g1_r28_a1f.gif},
\end{center} 
one can visualize a very asymmetric convergence towards the Sato--Tate measure in the case of Elkies' elliptic curve (the highest rank elliptic curve known to date, of rank at least 28).}

The proof of Theorem \ref{theorem: main} occupies \S\ref{section: proof}. It relies on work of Akbary, Ng, and Shahabi \cite{ANS14} on the existence of a limiting distribution attached to any unitary cuspidal automorphic selfdual $L$-function, and it follows the ideas that Sarnak introduced in \cite{Sar07} (see also the discussion of Mazur--Stein \cite[\S 9,\S 10]{mazur-stein} and Fiorilli \cite{Fiorilli}) to explain the bias in the sign of the Frobenius traces of an elliptic curve in terms of the rank. Sarnak restricted his attention to elliptic curves without complex multiplication, in which case $G=\USp(2)$ and the nontrivial irreducible representations of~$G$ are the symmetric powers of its standard representation.

In \S\ref{section: bounrank}, we use Theorem \ref{theorem: main} to give a simple upper bound of the asymptotic $L^2$-norm $I(\varphi)$. This is used in \S\ref{section: comments} to analyze the convergence rate of certain virtual characters which are of interest in the numerical calculation of Sato--Tate groups. Indeed, the usual method for such calculations is to compute approximations to the limit values of a certain set of virtual characters, for which these limit values are known to determine the Sato--Tate group. A typical example are the moments of the coefficients of the polynomial $L_\fp(A,T/\sqrt{|\fp|})$ (cf. \cite{KS09, FKRS12}) or the so-called power sums (cf. \cite{KS09}). Formula \eqref{equation: main} can then be used to determine how fast one can expect different families of virtual characters to converge. For instance, Shieh \cite{Shi16}  experimentally observed a better convergence of the set of irreducible characters than of the moments. This is consistent with the velocities of convergence predicted by \eqref{equation: main}, and in fact one can view Theorem~\ref{theorem: main} as a theoretical justification for Shieh's observation. 

In \S\ref{section: examples} we report on some numerical experiments, carried out to test whether the predictions of convergence rate arising from Theorem \ref{theorem: main} can actually be observed in concrete examples. The concrete choice of examples also reflects the goal of illustrating how the different invariants of abelian varieties (the various analytic ranks, the Frobenius--Schur index, the Sato--Tate group, etc.) affect the convergence rate.

In the spirit of Sarnak's letter \cite{Sar07}, the results of \S\ref{section: proof} can be used to study the bias of the sign of the Frobenius trace of an abelian variety. We briefly report on this in \S\ref{section: xevixef}.

\textbf{Notations.} We use $|\cdot|$ to denote the absolute norm of ideals in rings of integers of number fields and the complex absolute value. Its use in one sense or the other should be clear from the context.

\textbf{Acknowledgments.} Guitart is thankful to the ESAGA group in the University of Duisburg--Essen, for their warm hospitality during his visit on the spring of $2016$ where part of this work was carried out. Thanks to Amir Akbary, Jorge Jim\'enez, Kiran S. Kedlaya, Victor Rotger, and Mark Watkins for suggestive remarks and helpful comments.

\section{Background}\label{section: background}

In this section, we recall the results of \cite{ANS14} that we will require in \S\ref{section: proof} and a few technical lemmas. 

\subsection{Limiting distributions} 

Let $\psi\colon \R_{>0}\ra \R$ be a function admitting an expression of the form
\begin{equation}\label{equation: special form}
\psi(x)=c+S(x,T)+R(x,T)\,,
\end{equation}
for any $T\geq 2$, where $c\in\R$ and: 
\begin{enumerate}[i)]
\item The main term $S(x,T)$ is of the form $\Re\left(\sum_{\gamma_n\leq T}\eta_nx^{i\gamma_n} \right)$, where $\{\gamma_n\}_{n\geq 1}\subseteq \R_{>0}$ is a non-decreasing sequence that tends to infinity, $\{\eta_n\}_{n\geq 1}\subseteq \C$, and there exists $\theta \in [0,3-\sqrt 3]$ such that
\begin{equation}\label{equation: cond 1}
\sum_{\gamma_n\leq T}\gamma_n^2|\eta_n|^2=O(T^\theta)\,.
\end{equation}
\item The error term $R(x,T)$ satisfies
\begin{equation}\label{equation: cond 2}
\lim_{X\ra \infty}\frac{1}{\log(X)}\int_2^X|R(x,X)|^2\frac{dx}{x}=0\,.
\end{equation}
\end{enumerate}
The next statement is contained in \cite[Cor. 1.3]{ANS14}.

\begin{theorem}[\cite{ANS14}]\label{theorem: background 2}
The function $\psi$ possesses a limiting distribution $\mu_\psi$ with respect to the measure $dx/x$. That is, for any continuous function $f: \R \ra \R$ one has
\begin{equation}\label{equation: loglim}
\lim_{X\ra \infty} \frac{1}{\log(X)}\int_2 ^Xf(\psi(x))\frac{dx}{x}=\int_\R f(x) \mu_\psi(x)\,.
\end{equation}
Moreover, the expectation and variance of~$\mu_\psi$ are respectively
$$
\mathrm E[\mu_\psi]:=\int_\R x \mu_\psi(x)= c\,,\qquad \mathrm V[\mu_\psi]:=\int_\R (x-c)^2\mu_\psi(x)=\frac{1}{2}\sum_{n\geq 1}|\eta_n|^2\,.
$$
\end{theorem}

\begin{remark}\label{remark: opetitasqrtx}
Let $\psi$ be as in \eqref{equation: special form} and let $\psi':\R_{>0}\ra \R$ be such that $\psi(x)=\psi'(x)+\Delta(x)$, where $\Delta(x)=o(1)$. If we write $\psi'(s)$ as
$$
\psi'(x)=c+S(x,T)+R'(x,T)\,,
$$
the main term $S(x,T)$ and the error term $R'(x,T):=R(x,T)-\Delta(x)$ also satisfy \eqref{equation: cond 1} and \eqref{equation: cond 2}. Indeed, this amounts to showing that
$$
\lim_{X\ra\infty}\frac{1}{\log(X)}\int_2^X |\Delta(x)|^2\frac{dx}{x}=0\,.
$$
To prove this, let $\epsilon>0$ and note that, since $\Delta(x)=o(1)$, there exists $X_\epsilon>0$ such that $|\Delta(x)|^2<\epsilon$ for $x>X_\epsilon$. Then
$$
\lim_{X\ra\infty}\frac{1}{\log(X)}\int_{2}^X |\Delta(x)|^2\frac{dx}{x}=\lim_{X\ra\infty}\frac{1}{\log(X)}\int_{X_\epsilon}^X |\Delta(x)|^2\frac{dx}{x}\leq\lim_{X\ra\infty}\frac{\epsilon}{\log(X)}\int_{X_\epsilon}^X \frac{dx}{x}\leq \epsilon\,.
$$
\end{remark}

We finish this section by recording a property of the limit appearing on the left hand side of \eqref{equation: loglim} that we will use in \S\ref{section: proof}.

\begin{lemma}\label{lemma: prolim}
  Let $f\colon \R_{>0}\ra \R_{\geq 0}$ be a locally integrable function such that
  $$
  \lim_{X\ra \infty}\frac{1}{\log (X)}\int_2^X f(x)dx=C\,,
  $$ 
  for some $C\in\R$. Let $g\colon \R_{>0}\ra \R$ be a function such that $\displaystyle\lim_{x\ra\infty } g(x)= 1$ and such that $fg$ is locally integrable. Then
  \begin{align*}
    \lim_{X\ra \infty}\frac{1}{\log (X)}\int_2^X f(x)g(x)dx=C\,.
  \end{align*}
\end{lemma}
\begin{proof}
  First of all, we observe that for any $M\geq 2$ we have that
  \begin{align}\label{eq: 10}
    \lim_{X\ra \infty}\frac{1}{\log (X)}\int_M^X f(x)dx=C.
  \end{align}
Now fix $\epsilon>0$ and let $X_\epsilon$ be such that $g(x)$ belongs to the interval $[1-\epsilon,1+\epsilon]$ for all $x\geq X_\epsilon$. Then, for any $X\geq X_\epsilon$ we have
\begin{align*}
  \frac{1}{\log (X)}\int_2^X f(x)g(x)dx =&   \frac{1}{\log (X)}\int_2^{X_\epsilon} f(x)g(x)dx +   \frac{1}{\log (X)}\int_{X_\epsilon}^X f(x)g(x)dx \\
\leq & \frac{1}{\log (X)}\int_2^{X_\epsilon} f(x)g(x)dx +   \frac{(1+\epsilon)}{\log (X)}\int_{X_\epsilon}^X f(x)dx.
\end{align*}
Now, by \eqref{eq: 10} for $X$ sufficiently large we will have that
\begin{align*}
   \frac{1}{\log (X)}\int_{X_\epsilon}^X f(x)dx \leq C + \epsilon\,.
\end{align*}
Also, for $X$ large enough 
\begin{align*}
  \frac{1}{\log (X)}\int_2^{X_\epsilon} f(x)g(x)dx \leq \epsilon\,.
\end{align*}
Therefore, we see that for $X$ large enough we have that
\begin{align*}
  \frac{1}{\log (X)}\int_2^X f(x)g(x)dx \leq \epsilon + (1+\epsilon)(C+\epsilon)=C + \epsilon (2+C+\epsilon).
\end{align*}
Now, mutatis mutandis one can also show that for $X$ large enough one has
\begin{align*}
  \frac{1}{\log X}\int_2^X f(x)g(x)dx \geq C - \epsilon (2+C-\epsilon)
\end{align*}
and we see that the limit when $X\ra \infty$ is also $C$.
\end{proof}


\subsection{Limiting distributions of automorphic $L$-functions}

Let $\pi$ be an irreducible unitary cuspidal automorphic representation of $\GL_r(\A_\Q)$, for some $r\geq 1$, and let $L(\pi,s)$ denote the automorphic $L$-function attached to $\pi$. Suppose that $L(\pi,s)\not=\zeta(s-i\tau_0)$ for any $\tau_0\in \R$, where~$\zeta$ denotes the Riemann Zeta function. 
For $n\geq 1$, define the coefficients $\Lambda_\pi(n)$, by prescribing an equality of Dirichlet series
\begin{equation}
  -\frac{L(\pi,s)'}{L(\pi,s)}=:\sum_{n\geq 1}\Lambda_{\pi}(n) n^{-s}\,,
\end{equation}
and, for $x>0$, define the function 
\begin{align}\label{eq: automorphicpsi}
\psi(\pi,x):=\frac{1}{\sqrt x}\sum_{n\leq x}\Lambda_\pi(n)\,.
\end{align}
The following is \cite[Prop. 4.2]{ANS14}.

\begin{theorem}\label{theorem: background 1} Under the Riemann Hypothesis for $L(\pi,s)$, for any $x>0$ and $T\geq 2$, we have 
\begin{align}
  \psi(\pi,x)= -2r_{\pi} - \sum_{0<|\gamma_\pi|\leq T} \frac{x^{i\gamma_\pi}}{1/2+i\gamma_\pi} + R_\pi(x,T)\,.
\end{align}
Here, $r_\pi$ denotes the order of the zero of $L(\pi,s)$ at $s=1/2$, $\gamma_\pi$ runs over the non-zero imaginary parts of absolute value up to $T$ of the zeros of $L(\pi,s)$ on the critical line, and the error term $R_\pi(x,T)$ satisfies
\begin{align*}
  R_\pi(x,T) = O\left( \frac{x^{1/2+\beta}\log^2(x)}{T} + x^{\beta-1/2}\log (x)+\frac{x^{1/2}\log^2 (T)}{T\log (x)}+\frac{x^{1/2}\log (T)}{T}\right),
\end{align*}
for some $\beta \in [0,1/2)$.
\end{theorem}

\begin{remark}\label{remark: ontheproof}
By combining Theorems~\ref{theorem: background 1} and \ref{theorem: background 2}, one obtains as in \cite[\S4.1]{ANS14} that if $\pi$ is selfdual\footnote{We alert the reader that the selfduality condition on $\pi$ does not appear in the statement of \cite[Cor. 1.5]{ANS14}, although it is used in the proof (see the last display in page $767$ of loc. cit.).}, then $\psi(\pi,x)$ has a limiting distribution with respect to the measure $dx/x$ (see \cite[Cor. 1.5]{ANS14}). Indeed, if $\pi$ is selfdual the zeros of $L(\pi,s)$ come in conjugate pairs, and by pairing them one sees that $\psi(\pi,x)$ can be written as in \eqref{equation: special form} in such a way that $R_\pi(x,T)$ satisfies \eqref{equation: cond 2} and that, if $\{\gamma_n\}_{n\geq 1}$ is an ordering by size of the positive imaginary parts of the zeros of $L(\pi,s)$ on the critical line, then \eqref{equation: cond 1} is satisfied with $\eta_n:=-2/(1/2+i\gamma_n)$ and $\theta:=\beta$. We will apply a similar argument to $c\psi(\pi,x)+\bar c\psi(\pi^\vee,x)$, with $c\in \C$, if $\pi$ is not selfdual (see Corollary~\ref{corollary: psi and bar psi} below).
\end{remark}

\section{Proof of Theorem \ref{theorem: main}}\label{section: proof}

Resume the notations of \S\ref{section: intro}. The first step towards the proof of Theorem~\ref{theorem: main} is to consider the case of a nontrivial irreducible character $\chi$ of $G$. 
For $x>0$, define the function
$$
\psi(\chi,x):=\frac{\log( x)}{\sqrt x}\sum_{|\fp|\leq x}\chi(y_\fp)\,.
$$

\begin{proposition}\label{proposition: irred chars} Under Assumption \ref{assumption: main}, for a nontrivial irreducible character $\chi$ of $G$ and for any $x>0$ and $T\geq 2$, we have that
\begin{align}\label{equation: psi(chi,x)}
\psi(\chi,x)=-2r_\chi-u_\chi- \sum_{0<|\gamma_\chi|\leq T}\frac{x^{i\gamma_\chi}}{1/2+i\gamma_\chi} +R(x,T)\,,
\end{align}
where $r_\chi$, $u_\chi$ are as in the statement of Theorem~\ref{theorem: main}, $\gamma_\chi$ runs over the non-zero imaginary parts of absolute value up to $T$ of the zeros of $L(\chi,s)$ on the critical line, and the error term $R(x,T)$ satisfies \eqref{equation: cond 2}.
\end{proposition}

\begin{proof}
Let $\varrho$ be an irreducible representation of $G$ of character $\chi$, and let $d_\chi$ denote the degree of~$\varrho$. For $n\geq 1$, define the von-Mangold function $\Lambda_\chi(n)$ by prescribing an equality of Dirichlet series
\begin{equation}\label{equation: coefsdirser}
-\frac{L(\chi,s)'}{L(\chi,s)}=:\sum_{n\geq 1}\Lambda_{\chi}(n) n^{-s}\,.
\end{equation} 
For a prime $\fp$ of $k$, let $\alpha_{\fp,1},\dots, \alpha_{\fp,d_{\chi}(\fp)}$ denote the recyprocal roots of $L_\fp(\chi,T)$. Taking the logarithmic derivative in \eqref{equation: compL}, we find that
\begin{align}\label{eq: log der}
 -\frac{L(\chi,s)'}{L(\chi,s)} = \sum_{\fp} \sum_{j=1}^{d_\chi(\p)}\sum_{r\geq 1}\frac{\log (|\fp|) \alpha_{\p,j}^r}{|\fp|^{rs}}=\sum_{r\geq 1} \sum_{\fp} \log (|\fp|)\big(\sum_{j=1}^{d_\chi(\p)}\alpha_{\p,j}^r\big)|\fp|^{-rs}\,.
\end{align}
For $r\geq 1$, define\footnote{Note that we are using the same notation $\Lambda_\chi$ to denote two different functions; since one is a function of the positive integers and the other is a function of the ideals of $k$, their argument makes clear which one we refer to (when $k=\Q$ the two functions coincide).}
$$
\Lambda_{\chi}(\fp^r):=\log (|\fp|)\big(\sum_{j=1}^{d_\chi(\p)}\alpha_{\p,j}^r\big)\,.
$$
Note that, if $\fp\in P$, then $\alpha_{\fp,1},\dots, \alpha_{\fp,d_{\chi}(\fp)}$ are the eigenvalues of $\varrho(y_\fp)$ and thus $\Lambda_\chi(\fp^r)=\log (|\fp|)\chi(y_\fp^r)$. Observe also that we similarly find that 
$$
-\frac{L^P(\chi,s)'}{L^P(\chi,s)} =\sum_{r\geq 1} \sum_{\fp\in P} \Lambda_\chi(\fp^r)|\fp|^{-rs}\,,
$$
For $x>0$, define the function
\begin{equation}\label{eq: explicit formula}
\psi_{1}(\chi,x):=\frac{1}{\sqrt x}\sum_{n\leq x}\Lambda_\chi(n)=\frac{1}{\sqrt{x}}\sum_{r\geq 1}\sum_{|\fp|^r\leq x}\Lambda_\chi(\fp^r)\,,
\end{equation}
where the second equality follows from compairing \eqref{eq: log der} and \eqref{equation: coefsdirser}. Under Assumption~\ref{assumption: main}, there is an irreducible unitary cuspidal representation $\pi$ of $\GL_{d_\chi[k:\Q]}(\A_\Q)$ such that $\psi_{1}(\chi,x)$ coincides with $\psi(\pi,x)$, as defined in \eqref{eq: automorphicpsi} (note that $L(\pi,s) $ is not of the form $\zeta(s-i\tau_0)$ since $\chi$ is nontrivial). Then, by Theorem \ref{theorem: background 1}  we can write 
$$
\psi_{1}(\chi,x)=c_{1}+ S_{1}(x,T)+R_{1}(x,T)\,,
$$ 
where $c_{1}:=-2r_\chi$, $S_{1}(T):=-\sum_{0<|\gamma_\chi|\leq T}\frac{x^{i\gamma_\chi}}{1/2+i\gamma_\chi}$  and $R_{1}(x,T)$ satisfies \eqref{equation: cond 2}.
Define
$$
\psi_2(\chi,x):=\frac{1}{\sqrt x}\sum_{|\fp|\leq x}\Lambda_\chi(\fp)\,,\qquad \Delta_1(x):=\frac{1}{\sqrt x}\left( \sum_{|\fp|\leq x^{1/2}} \Lambda_{\chi}(\fp^2) + \sum_{k\geq 3}\sum_{|\fp|\leq x^{1/k}}\Lambda_{\chi}(\fp^k)\right)\,,
$$
so that $\psi_1(\chi,x)= \psi_2(\chi,x)+\Delta_1(x)$. We proceed to study the size of $\Delta_1(x)$. On the one hand, we have 
\begin{align*}
  \sum_{|\fp|\leq x^{1/k}}\Lambda_\chi(\fp^k)\leq d_\chi \sum_{|\fp|\leq x^{1/k}}\log(|\fp|) \leq d_\chi[k:\Q] x^{1/k}\log(x)\,.
\end{align*}
Since there is no prime ideal $\fp$ of~$k$ such that $|\fp|\leq x^{1/k}$ when $k>\log_2(x)$, the sum indexed by~$k$ in $\Delta_1(x)$ has at most $O(\log (x))$ summands. Therefore the right most summand in $\Delta_1(x)$ is $O(x^{1/3}\log^2 (x))$.
On the other hand, \eqref{equation: equidtrick} says that
$$
\sum_{|\fp|\leq x^{1/2}} \Lambda_{\chi}(\fp^2)= \delta (\chi(\cdot^2))\sqrt x+ o(\sqrt x),
$$
where $\chi(\cdot^2)$ denotes the central function $g\mapsto \chi(g^2)$. Note that the multiplicity $\delta (\chi(\cdot^2))$ equals $u_\chi:=\int_G \chi(g^2)\mu_G(g)$, the so-called Frobenius--Schur index of $\chi$. It is known that $u_\chi=0$ if $\chi$ takes some complex nonreal value, $u_\chi=1$ if $\chi$ is attached to a representation realizable over $\R$, and $u_\chi=-1$ if $\chi$ is attached to a quaternionic representation (i.e., a representation that is not realizable over $\R$, even if its character takes values in $\R$).
We deduce that 
$$
\psi_2(\chi,x)= \psi_1(\chi,x)-u_\chi+ o(1)\,,
$$
and by Remark~\ref{remark: opetitasqrtx}, we deduce that 
\begin{equation}\label{equation: psi1}
\psi_2(\chi,x)=c_2+S_2(x,T)+R_2(x,T)\,,
\end{equation} 
where $c_2:=c_1-u_\chi$, $S_2(x,T):=S_1(x,T)$, and $R_2(x,T)$ satisfies \eqref{equation: cond 2}.
For $x>0$, define the function
$$
\psi_3(\chi,x):=\frac{1}{\sqrt{x}}\sum_{|\fp|\leq x,\p\in P}\Lambda_\chi(\fp)\,,
$$
so that we have that 
$$
\psi_2(\chi,x)=\psi_3(\chi,x)+\Delta_2(x)\,,\qquad\text{where}\qquad\Delta_2(x):=\frac{1}{\sqrt x}\sum_{\p|\fN}\Lambda_\chi(\fp)\,. 
$$
Provided that there is only a finite number of primes dividing $\fN$, we have that $\Delta_2(x)=o(1)$.
But, by the Abel summation trick (see Lemma~\ref{lemma: abel summation} below), we have that $\psi_3(\chi,x)=\psi(\chi,x)+o(1)$. It follows from Remark~\ref{remark: opetitasqrtx}, that $\psi(\chi,x)$ admits an expression of the form $\psi(\chi,x)=c+S(x,T)+R(x,T)$, with $c=c_2$, $S(x,T):=S_2(x,T)$, and $R(x,T)$ satisfies \eqref{equation: cond 2}. 

\end{proof}
 
At the end of the previous proof we have applied the Abel summation trick. Following the lines of reasoning of \cite[Lemma 2.1]{RS}, we give the argument involved in gory detail.
\begin{lemma}\label{lemma: abel summation}
For a nontrivial character $\chi$ of $G$, we have 
  \begin{align}
   \sum_{|\fp|\leq x}\log(|\fp|)\chi(y_\fp)= \log(x)\sum_{|\fp|\leq x}\chi(y_\fp)+o(\sqrt x)\,.
  \end{align}
\end{lemma}
\begin{proof}
Let us write 
$$
\vartheta(x):=\sqrt x\cdot \psi_3(\chi,x)=\sum_{|\fp|\leq x}\log(|\fp|)\chi(y_\fp)\,.
$$
Applying summation by parts\footnote{We apply the formula of \cite[Chap. I, \S 1.5]{IK}, with $f(p^r)=\sum_{|\fp|=p^r}\log(|\fp|)\chi(y_\fp)$ (and $f(n)=0$ if $n$ is not a prime power) and $g(x)=1/\log(x)$.}, we have
\begin{align}\label{eq: 8}
  \sum_{|\fp|\leq x}\chi(y_\fp)=\frac{\vartheta(x)}{\log(x)}+\int_{2}^x \frac{\vartheta(t)}{t\log^2 (t)}dt\,.
\end{align}
Then, integrating \eqref{equation: psi1} times $\sqrt x$,  letting $T$ tend to infinity, and using that $\beta<1/2$ we obtain 
\begin{align}
 G(x):=\int_{2}^x \vartheta(t)dt=\frac{2}{3}(-2r_{\chi}-u_{\chi})x^{3/2} - \sum_{\gamma_{\chi}\neq 0}\frac{x^{\frac{3}{2}+i\gamma_{\chi}}}{(\frac{1}{2}+i\gamma_{\chi})(\frac{3}{2}+i\gamma_{\chi})}+o(x^{3/2}).
\end{align}
As in the proof of \cite[Lemma 2.1]{RS}, here we are using the fact that the series over~$\gamma_{\chi}$ converges absolutely. This follows from the asymptotic formula for the number $N(T)$ of zeros on the critical line with imaginary part of absolute value $\leq T$ for automorphic $L$-functions, as in for example \cite[Theorem 5.8]{IK}. It follows that $G(x)=O(x^{3/2})$.

Integrating by parts the rightmost integral in \eqref{eq: 8}, we obtain
$$
\int_{2}^x \frac{\vartheta(t)}{t\log^2 (t)}dt=\frac{G(x)}{x\log^2 x}+\int_2^xG(t)\frac{\log(t)+2}{t^2\log^3(t)}dt= O\left(\frac{\sqrt x}{\log^2(x)}\right)+O\left( \int_2 ^x\frac{dt}{\sqrt t\cdot \log^2(t)}\right)\,. 
$$ 
Both $O$-terms in the above expression are $o\left(\frac{\sqrt{x}}{\log x}\right)$ (the first by trivial reasons and the second, by l'H\^opital, for example). We conclude that
\begin{align}
    \sum_{|\fp|\leq x} \chi (y_\fp)=\frac{\vartheta(x)}{\log(x)}+o\left(\frac{\sqrt{x}}{\log (x)}\right),
\end{align}
and this implies the lemma.
\end{proof}
\begin{corollary}\label{corollary: psi and bar psi}
  With the notation and assumptions as in Proposition \ref{proposition: irred chars} and for $c_\chi\in\C$ we have that 
  \begin{align*}
   c_\chi \psi(\chi,x)+\bar c_\chi \psi(\bar\chi,x)= -2(c_\chi r_\chi+\bar c_\chi r_{\bar\chi}) -  (c_\chi u_\chi+\bar c_\chi u_{\bar\chi}) +S(x,T) + R(x,T),
  \end{align*}
where $S(x,T)$ and $R(x,T)$ are as in \eqref{equation: special form}, $R(x,T)$ satisfies \eqref{equation: cond 2}, and $S(x,T)$ satisfies \eqref{equation: cond 1} with $\{\gamma_n\}_{n\geq 1}$ being the positive imaginary parts of the zeros of $L(\chi,s)$ and $L(\bar\chi,s)$, $\eta_n:=\frac{-2c_\chi}{1/2+i\gamma_n}$ if $\gamma_n$ corresponds a zero of $L(\chi,s)$, $\eta_n:=\frac{-2\bar c_\chi}{1/2+i\gamma_n}$ if $\gamma_n$ corresponds to a zero of $L(\bar\chi,s)$, and $\theta=1/2$.
\end{corollary}
\begin{proof}
  This follows at once by adding the expressions for $c_\chi \psi(\chi,x)$ and $\bar c_\chi \psi(\bar\chi,x)$ given in \eqref{equation: psi(chi,x)} and pairing each zero $\frac{1}{2}+i\gamma_\chi$ of $L(\chi,s)$ having $\gamma_\chi>0$ with its conjugate, which is a zero of $L(\bar\chi,s)$ (and, similarly, pairing each zero $\frac{1}{2}+i\gamma_{\bar\chi}$ of $L(\bar\chi,s)$ having $\gamma_{\bar\chi}>0$ with its conjugate, which is a zero of $L(\chi,s)$).
\end{proof}
\begin{remark}
  Observe that, in fact, $r_\chi =r_{\bar \chi}$ and $u_\chi = u_{\bar \chi}$.
\end{remark}

\begin{theorem}
Under Assumption~\ref{assumption: main}, for every virtual character of the form $\varphi=\sum_{\chi\not =1} c_\chi \chi$, where $c_\chi\in \C$ and $\chi$ runs over the irreducible nontrivial characters of $G$, one has
\begin{equation}\label{eq:310}
  \lim_{X\ra \infty} \frac{1}{\log X}\int_2^X \psi(\varphi,x)\frac{dx}{x} = - \sum_{\chi}c_\chi(2r_\chi+u_\chi).
\end{equation}

\begin{equation}\label{eq:311}
  \lim_{X\ra \infty} \frac{1}{\log X}\int_2^X |\psi(\varphi,x)|^2\frac{dx}{x}=\big|\sum_{\chi}c_\chi(2r_\chi+u_\chi)\big|^2+\sum_\chi\sum_{\gamma_\chi\neq 0}\frac{|c_\chi|^2}{1/4+\gamma_\chi^2}\,,
\end{equation}



where $\psi(\varphi,x):=\sum_{\chi}c_\chi \psi(\chi,x)$ and $r_\chi$, $u_\chi$, and $\gamma_\chi$ are as in the statement of Theorem \ref{theorem: main}.
\end{theorem}
\begin{proof}
We begin by assuming that $\varphi$ is selfdual. We say that a virtual character $\varphi = \sum_{\chi}c_\chi \chi$, where $c_\chi\in \C$ and $\chi$ runs over irreducible characters of $G$, is selfdual if\footnote{Observe that this condition is trivially satisfied if $\chi$ is selfdual, since $\chi\simeq\bar\chi$ in this case.} $c_\chi=\bar c_{\bar\chi}$, where $\bar\chi$ denotes the dual (or complex conjugate) of $\chi$. 

By Corollary \ref{corollary: psi and bar psi} and Theorem \ref{theorem: background 2}, we have that for each of the irreducible constituents $\chi$ of $\varphi$, the function $c_\chi \psi(\chi,x)+c_{\bar \chi}\psi(\bar\chi,x)$ admits a limiting distribution $\mu_\chi$ with
$$
\mathrm E[\mu_\chi]=-2(c_{\chi} r_\chi+c_{\bar \chi}r_{\bar\chi})-(c_{\chi} u_\chi+c_{\bar \chi}u_{\bar\chi})$$

 $$
 \mathrm V[\mu_\chi]= \sum_{\gamma_\chi>0}\frac{{2}|c_\chi|^2}{1/4+\gamma_\chi ^2}+\sum_{\gamma_{\bar\chi}>0}\frac{{2}|\bar c_\chi|^2}{1/4+\gamma_{\bar\chi}^2} .
$$
 By additivity and the fact that $\varphi$ is selfdual, it follows again from Theorem \ref{theorem: background 2} that the function $\psi(\varphi,x)$ admits a limiting distribution $\mu_\varphi$ with
$$
\mathrm E[\mu_\varphi]=-\sum_{\chi}c_\chi(2r_\chi+u_\chi)\,,\qquad \mathrm V[\mu_\varphi]= {2}\sum_{\chi}\sum_{\gamma_\chi>0}\frac{|c_\chi|^2}{1/4+\gamma_\chi ^2}\,.
$$
Now \eqref{eq:310} is a restatement of the above formula for the mean and \eqref{eq:311} is obtained as
$$
\lim_{X\ra \infty}\frac{1}{\log (X)}\int _2 ^X\psi(\varphi,x)^2\frac{dx}{x}=\mathrm E[\mu_\varphi]^2+\mathrm V[\mu_\varphi]=\left(\sum_{\chi}c_\chi(2r_\chi+u_\chi)\right) ^2+{2}\sum_{\chi}\sum_{\gamma_\chi>0}\frac{|c_\chi| ^2}{1/4+\gamma_\chi ^2}\,.
$$
Now consider the case where $\varphi$ is not selfdual. Then we can write $\varphi = \varphi_R+i\varphi_I$, where $$\varphi_R := \frac{\varphi+\bar\varphi}{2},\ \ \text{ and } \ \ \varphi_I := \frac{\varphi-\bar\varphi}{2i} .$$
One easily checks that $\varphi_R$ and $\varphi_I$ are virtual selfdual characters, and this immediately implies \eqref{eq:310}. As for \eqref{eq:311}, note that one has
\begin{align*}
  |\psi(\varphi,x)|^2=\psi(\varphi_R,x)^2+\psi(\varphi_I,x)^2.
\end{align*}
Applying the previous case to both terms, and noting that 
\begin{align*}
  \varphi_R = \sum_{\chi}\frac{c_\chi+\bar c_{\bar\chi}}{2}\chi, \ \   \varphi_I = \sum_{\chi}\frac{c_\chi-\bar c_{\bar\chi}}{2i}\chi,
\end{align*}
we get
\begin{align*}
  \lim_{X\ra \infty} \frac{1}{\log X}\int_2^X |\psi(\varphi,x)|^2\frac{dx}{x}=\big(\sum_{\chi}\frac{c_\chi+\bar c_{\bar\chi}}{2}(2r_\chi+u_\chi)\big)^2+\frac{1}{2}\sum_\chi\sum_{\gamma_\chi> 0}\frac{|c_\chi+\bar c_{\bar\chi}|^2}{1/4+\gamma_\chi^2} \\ -  \big(\sum_{\chi}\frac{c_\chi-\bar c_{\bar\chi}}{2}(2r_\chi+u_\chi)\big)^2+\frac{1}{2}\sum_\chi\sum_{\gamma_\chi> 0}\frac{|c_\chi-\bar c_{\bar\chi}|^2}{1/4+\gamma_\chi^2}.
\end{align*}
The first and third term in the right hand side of the above expression add up to $$\big|\sum_{\chi}c_\chi(2r_\chi+u_\chi)\big|^2,$$ and the second and fourth, applying parallelogram's law, add up to
\begin{align*}
  \sum_\chi \sum_{\gamma_\chi>0}\frac{|c_\chi|^2}{\frac{1}{4}+\gamma_\chi^2} +  \sum_\chi \sum_{\gamma_\chi>0}\frac{| c_{\bar \chi}|^2}{\frac{1}{4}+\gamma_\chi^2};
\end{align*}
Now using that
\begin{align*}
  \sum_\chi \sum_{\gamma_\chi>0}\frac{| c_{\bar \chi}|^2}{\frac{1}{4}+\gamma_\chi^2} = \sum_\chi \sum_{\gamma_{\bar\chi>0}}\frac{| c_{ \chi}|^2}{\frac{1}{4}+\gamma_{\bar\chi}^2}= \sum_\chi \sum_{\gamma_\chi<0}\frac{| c_{ \chi}|^2}{\frac{1}{4}+\gamma_\chi^2}
\end{align*}
we conclude the proof.
\end{proof}
We can now prove Theorem~\ref{theorem: main}. Indeed, it follows from \eqref{eq:311}, Lemma~\ref{lemma: prolim} and the fact that 
$$
\frac{|\psi(\varphi,x)|^2}{x}\sim |\delta(\varphi,x)|^2\qquad \text{as}\qquad x\ra \infty\, .$$

\section{An upper bound for the asymptotic $L^2$-norm}\label{section: bounrank}

In this section we elaborate on Theorem~\ref{theorem: main} to obtain an explicit and simple upper bound for the asymptotic $L^2$-norm, which we will use in the applications discussed in \S\ref{section: comments}. Since in \S\ref{section: comments} we will only be interested in selfdual virtual characters with integral coefficients, from now on we will make the simplifying assumption that $\varphi\in\bigoplus_{\chi\neq 1}\Z \cdot \chi$ and $c_\chi=c_{\bar\chi}$. We remark that all the constants appearing in this section are absolute and effectively computable.

We will need the following consequence of Assumption~\ref{assumption: main} (1), which we record as a remark for future reference.

\begin{remark}\label{remark: completedL} If $\chi$ is an irreducible nontrivial character of $G$, as a consequence of the fact that $L(\chi,s)$ is automorphic, we have that (see \cite[\S5]{IK}) there exist a positive integer $N_\chi$ (the so-called \emph{absolute conductor} of $\chi$) and complex numbers $\kappa_{\chi,j}$ with $\Re(\kappa_{\chi,j})>-1$ for $j=1,\dots, d_\chi$ (the so-called \emph{local parameters at infinity} of $L(\chi,s)$), such that the completed $L$-function
\begin{equation}\label{equation: completedL}
\Lambda(\chi,s):=N_\chi^{s/2}\Gamma(\chi,s)L(\chi,s)\,,\qquad \text{with }\Gamma(\chi,s):=\pi^{-d_\chi s/2}\prod_{j=1}^{d_\chi}\Gamma\left(\frac{s+\kappa_{\chi,j}}{2}\right)\,,  
\end{equation}
defined for  $\Re(s)>1$, extends to an analytic function on $\C$ and satisfies a functional equation
\begin{equation}\label{equation: functional}
\Lambda(\chi,s)=\varepsilon_\chi\Lambda(\overline \chi,1-s)\,.
\end{equation}
Here, $\varepsilon_\chi\in \C$ satisfies $|\varepsilon_\chi|=1$, and 
\begin{equation}\label{equation: loc par bound}
|\kappa_{\chi,j}|\leq w_\chi \,,\qquad N_\chi \leq N^{d_\chi}\,,
\end{equation}
where $N:=|\fN|$ and $w_\chi$ is the weight\footnote{We define the weight $w_\chi$ of an irreducible character $\chi$ of $G$ as the weight of the $\ell$-adic representation $\chi\circ \varrho_A$.} of $\chi$. For the precise definitions of $N_\chi$ and $\kappa_{\chi,j}$, we refer to \cite{Ser69}. The inequalities in \eqref{equation: loc par bound} follow directly from the definitions. 
\end{remark}

\begin{proposition}\label{proposition: boundL2} Under Assumption~\ref{assumption: main}, there exists an absolute constant $K_1>0$ such that, for every virtual selfdual character $\varphi=\sum_{\chi\not =1} c_\chi \chi$ not containing the trivial character, one has 
$$
I(\varphi)\leq K_1 \left( \big(\sum_{\chi}|c_\chi| S_\chi\big)^2+\sum_{\chi}c_\chi^2S_\chi\right)\,,
$$
where
$$
S_\chi:=d_\chi[k:\Q]\log( N(w_\chi+3))\,.
$$

\end{proposition}

\begin{proof}
By Theorem~\ref{theorem: main}, we have  $I(\varphi)=I_1(\varphi)+I_2(\varphi)$, where 
\begin{equation}\label{equation: I1I2}
I_1(\varphi):=\left(\sum_\chi c_\chi(2r_\chi+u_\chi)\right)^2\,,\qquad I_2(\varphi):=\sum_\chi \sum_{\gamma_\chi>0}\frac{2c_\chi^2}{1/4+\gamma_\chi^2}\,.
\end{equation}
Note that $r_\chi\leq m(\chi,0)$, where $m(\chi,0)$ is as in Lemma~\ref{lemma: num zeros} below. The Lemma implies then that $I_1(\varphi)\leq K_2\cdot (\sum_\chi c_\chi S_\chi)^2$ for some $K_2> 0$. It suffices to prove that there exists $K_3> 0$ such that 
$$
\sum_{\rho_\chi}\frac{1}{|\rho_\chi|^2}\leq K_3S_\chi\,,
$$
where $\rho_\chi$ runs over the set zeros of $L(\chi,s)$ (equiv. $\Lambda(\chi,s)$) on the critical line.
Again by Lemma~\ref{lemma: num zeros} below, there exists $K_4>0$ such that
$$
\sum_{\gamma_\chi}\frac{1}{1/4+\gamma_\chi^2}\leq K_4\left(\sum_{m\geq 1}d_\chi\log\big(N (m+5)(w_\chi +3)\big)\frac{1}{1/4+(m-1)^2}\right)\,.
$$ 
and the proposition follows.
\end{proof}

The next result is well known. Provided that it is usually presented only in a form asymptotic in $T$ (which is not precise enough for our purposes), we have decided to include it in the form that we will require; that is, in the form of a statement valid for any $T\geq 0$. 

\begin{lemma}\label{lemma: num zeros}
Under Assumption~\ref{assumption: main}, for any irreducible character $\chi$ and any $T\geq 0$, the number $m(\chi,T)$ of zeros $\rho_\chi=1/2+\gamma_\chi$ (counted with multiplicity) of $\Lambda(\chi,s)$ with $|\gamma_\chi-T|\leq 1$ satisfies
$$
m(\chi,T)\leq K_4d_\chi[k:\Q]\log(N(T+5)(w_\chi+3))\,, 
$$
for an absolute constant $K_4>0$.
\end{lemma}

\begin{proof}
By taking the logarithmic derivative of Hadamard's factorization of $\Lambda(\chi,s)$ (see \cite[Thm. 5.6]{IK}) and using \cite[(5.29)]{IK}, one obtains 
$$
\Re\left( \frac{\Lambda(\chi,s)'}{\Lambda(\chi,s)}\right)=\sum_{\rho_\chi}\Re\left(\frac{1}{s-\rho_\chi}\right)\,, 
$$
for any $s\in \C$ distinct from a zero of $\Lambda(\chi,s)$. Thus, by taking $s=2+iT$, we get on the one hand
$$
\left | \frac{\Lambda(\chi,2+iT)'}{\Lambda(\chi,2+iT)}\right|\geq \sum_{\gamma_\chi}\frac{3/2}{(3/2)^2+(T-\gamma_\chi)^2}\geq \sum_{|\gamma_\chi-T|\leq 1}\frac{3/2}{(3/2)^2+(T-\gamma_\chi)^2}\geq \frac{6}{13}m(\chi,T)\,.
$$
On the other hand, by logarithmically differentiating \eqref{equation: completedL}, we obtain
$$
\left|\frac{\Lambda(\chi,2+iT)'}{\Lambda(\chi,2+iT)}\right|\leq\frac{1}{2}\log N_\chi+\left|\frac{\Gamma'(\chi,2+iT)}{\Gamma(\chi,2+iT)}\right|+\left|\frac{L'(\chi,2+iT)}{L(\chi,2+iT)}\right|\,.
$$ 
But 
$$
\left|\frac{L'(\chi,2+iT)}{L(\chi,2+iT)}\right|\leq d_\chi[k:\Q] \left| \frac{\zeta(2+iT)'}{\zeta(2+iT)}\right|
$$
and by \cite[(5.116)]{IK} and \cite[(5.8)]{IK}, we have that there exists $K_5>0$ such that
$$
\left|\frac{\Gamma'(\chi,2+iT)}{\Gamma(\chi,2+iT)}\right|\leq K_5\cdot \log(N_\chi\prod_{j=1}^{d_\chi}(|\kappa_{\chi,j}|+3)(T+5)^{d_\chi})\,.
$$
The lemma now follows from \eqref{equation: loc par bound} applied to the previous inequality.
\end{proof}

A direct application of the Cauchy-Schwarz inequality yields the following corollary. 

\begin{corollary}\label{corollary: upper bound}
Assume the hypotheses of Proposition~\ref{proposition: boundL2}. Let $S_\varphi$ be $\max_\chi\{S_\chi\}$, let $R_\varphi$ denote the number of irreducible constituents of $\varphi$, and write $C_\varphi:=\sum_\chi c_\chi^2$. Then there exists $K_6>0$ such that
$$
I(\varphi)\leq K_6 R_\varphi S_\varphi^2 C_\varphi\,.
$$
\end{corollary}

\section{An application to the numerical identification of Sato--Tate groups}\label{section: comments}

Let the notations be as in \S\ref{section: intro}. In particular, $A$ is an abelian variety defined over a number field $k$ of dimension~$g\geq 1$,~$G$ denotes the Sato--Tate group of $A$, and $\fp$ is a prime of $k$ of good reduction for $A$.

In recent years the development of fast methods for the computation of the polynomial $L_\fp(A,T)$ (as defined in \eqref{equation: Lpoly}) has made possible numerical approaches to the identification of the group~$G$, under the assumption of its conjectural equidistribution property (see \cite{FKRS12}, \cite{KS09} for example). 

Let us briefly describe these numerical approaches. One starts by selecting a family of virtual characters $\{\varphi_n\}_{n}$ of $\USp(2g)$ such that $\{\delta(\varphi_n|_G)\}_{n}$ identifies~$G$ as a subgroup up to conjugation inside $\USp(2g)$. For simplicity, throughout this section we will assume that $\varphi$ is selfdual and has integral coefficients. Then one can find an expression for $\varphi_n(y_\fp)$ in terms of $L_\fp(A,T)$ and thanks to the efficient methods of computation of $L_\fp(A,T)$, one can compute the sum
$$
\delta(\varphi_n,x):=\frac{1}{\pi(x)}\sum_{|\fp|<x}\varphi_n(y_\fp)
$$
for some large value of $x>0$. By \eqref{equation: char of equid}, we expect $\delta(\varphi_n,x)$ approach the multiplicity $\delta(\varphi_n|_G)$ of the trivial representation in $\varphi_n|_{G}$ as $x\ra \infty$.
If $X$ is large enough, the quantities $\delta(\varphi_n,X)$ will provide good approximations of the integer $\delta(\varphi_n|_G)$. 

By \cite[Rmk. 3.3]{FKRS12} it is expected that, for a fixed $g$, the set of possibilities for $G$ is finite. This is known to be true for $g\leq 3$ by Theorem 2.16 and Proposition 3.2 of \cite{FKRS12}. For example, if $g=1$ there are $3$ possibilities and if $g=2$ there are 52. Therefore, in practice the family $\{\varphi_n \}_n$ can be taken to be finite.

\begin{example}
Let $V$ denote the standard representation of $\USp(2g)$. For $n\geq 0$ and $0\leq k\leq 2g$, define the \emph{$n$-th moment of the $k$-th coefficient} as the character 
\begin{equation}\label{equation: nthmomentchar}
a_k^n:=\Tr\left((\Lambda^kV)^{\otimes n}\right)\,.
\end{equation}
Write simply $a_k:=a^1_k$. For $g=2$, one can see from Tables 9 and 10 of \cite{FKRS12} that the sequence of multiplicities\footnote{Note that $\delta(a_k^n)$ is denoted by $M_n[a_k]$ in \cite{FKRS12}.} $\{\delta(a_1^n|_G), \delta(a_2^n|_G)\}_{n}$ for $n=1,\dots, 7$ identifies the group~$G$. In \cite[\S5]{FKRS12}, the method described above for the family of coefficient moments $\{a_1^n,a_2^n\}_{n}$ is used to numerically identify Sato--Tate groups of several abelian surfaces.
\end{example}

One may wonder whether the efficiency of the above approach varies depending on the choice of the testing family of virtual characters. Shieh \cite{Shi16} has proposed to use the family of irreducible characters $\{\chi_n\}_{n}$ of $\USp(2g)$ instead of the family of coefficient moments $\{a_k^n\}_{n,k}$. 
Shieh applies the Brauer-Klymik formula to recover $\chi_n(y_\fp)$ from $L_\fp(A,T)$, and presents some numerical examples, where the family $\{\chi_n\}_{n}$ exhibits a much faster convergence than $\{a_k^n\}_{n,k}$. As remarked by Shieh (see comment at the last paragraph of \cite[\S1]{Shi16}), the difference between convergence rates of the two families is especially noticeable in the ``generic cases", that is, when $G=\USp(2g)$.

For a virtual character~$\varphi$ of $\USp(2g)$, write
\begin{equation}\label{equation: tilde}
\tilde\varphi:=\varphi|_G-\delta(\varphi|_G)\,.
\end{equation}
As discussed in \S\ref{section: intro}, the asymptotic $L^2$-norm $I(\tilde\varphi)$ can be seen as an estimate of the convergence rate of $\delta(\varphi,X)$ towards $\delta(\varphi|_G)$. Under this perspective, we can see Corollary~\ref{corollary: upper bound} as a justification of the efficiency of Shieh's proposal in the generic cases. Indeed, when $\chi_n$ is a nontrivial irreducible character of $\USp(2g)$ and $G=\USp(2g)$, the quantities $R_{\tilde\chi_n}=C_{\tilde\chi_n}=1$ are smallest possible. \footnote{However, we do not claim optimality of the family of irreducible characters (in terms of the velocity of convergence) among the class of families of central functions (see \S\ref{example: charnonoptimal}).}

We remark, however, that $\{R_{\tilde\chi_n}\}_n$ and $\{C_{\tilde\chi_n}\}_n$ can grow without bound when we take certain nongeneric $G\subseteq \USp(2g)$. We will now introduce a family of virtual characters $\{s_n\}_n$ of $\USp(2g)$, for which the sequence $\{C_{\tilde s_n}\}_n$ (and thus also the sequence $\{R_{\tilde s_n}\}_n$) stays bounded for every Sato--Tate group $G\subseteq\USp(2g)$.
Let $V$ denote the standard representation of $\USp(2g)$. For $n\geq 0$ and $0\leq k\leq 2g$, the \emph{$(n,k)$-th power sum} is the virtual character
\begin{equation}
s_n^k(\cdot):=\Tr(\Lambda^kV(\cdot ^n))\,.
\end{equation}
To ease notation, we simply write $s_n:=s_n^1$.
We remark that the family of virtual characters $\{s_n\}_n$ was previously considered in \cite{KS09}. 

\begin{proposition}\label{proposition: weights} 
Let $G$ denote the Sato--Tate group of an abelian variety $A$ of dimension $g$. Then, the sequence $\{C_{\tilde s_n}\}_n$ is bounded. 
\end{proposition}

\begin{proof}
Let $\mathfrak g$ denote the Lie algebra of $G$. Group representations of $G$ correspond to Lie algebra representations of $\mathfrak g$. For every $n\geq 0$, one can write $s_n$ as a finite sum of at most $2g$ weights of $\mathfrak g$. Since $\mathfrak g$ is semisimple, by \cite[Thm. 3.8]{Gup87} each weight of $\mathfrak g$ can be written as a bounded finite sum of distinct irreducible characters of $\mathfrak g$. Moreover, one easily sees from its description in \cite[Thm. 3.8]{Gup87}, that the multiplicity of every irreducible constituent of a weight is also bounded. The proposition follows.
\end{proof}

In the next section, we compare the convergence rates of the families $\{s_n\}_n$ and $\{a_1^n\}_n$ by computing the respective explicit bounds of Corollary~\ref{corollary: upper bound} for some of the Sato--Tate groups $G$ arising for $g=1$ and $g=2$.

\subsection{Rate of convergence of moments and power sums}

\subsubsection{Dimension $g=1$: non CM case}

The Sato--Tate group of an elliptic curve over~$k$ without CM is $G:=\USp(2)=\mathrm{SU}(2)$. If we denote by $V$ the standard representation of $G$, then the irreducible characters of $G$ are given by
\begin{equation}\label{equation: charUSp2}
\chi_n:= \Tr(\Sym^n(V))\qquad \text{for $n\geq 0$}\,.
\end{equation} 
For a virtual character $\varphi$ of $G$, let $\tilde\varphi$ be as in \eqref{equation: tilde}, and let $R_{\tilde\varphi}$, $C_{\tilde\varphi}$, and $S_{\tilde\varphi}$ be as defined in Corollary~\ref{corollary: upper bound}. 
For $n\geq 2$, easy computations writing $a_1^n$ and $s_n$ in terms of the $\chi_n$ show that $S_{\tilde a_1^n}=S_{\tilde s_n}$ and that
\begin{align}\label{eq: ex 511}
R_{\tilde a_1^n}=\left\lceil \frac{n}{2}\right\rceil\,,\,C_{\tilde a_1^n}=1+\sum_{0<j<n/2}\left(\binom{n}{j}-\binom{n}{j-1}\right)^2\,,\quad\text{while}\quad R_{\tilde s_n}=2\,,\,C_{\tilde s_n}=2\,.
\end{align}
Indeed, if $M\in G$ has eigenvalues $\alpha$ and $\bar\alpha$ then 
\begin{align*}
  a_1(M)= \alpha + \bar\alpha \ \ \ \text{ and } \ \  \chi_n(M)=\alpha^n + \alpha^{n-2} +\dots + \alpha^{-n}.
\end{align*}
Thus, for example, in the case of $a_1^n$ for even values of $n$ one has
\begin{align*}
a_1^n =  \chi_n + \left( \binom{n}{1} -1 \right)\chi_{n-2}  + \left( \binom{n}{2} -\binom{n}{1} \right)\chi_{n-4}+\dots + \left( \binom{n}{n/2} -\binom{n}{n/2-1} \right)\chi_{0}\, .
\end{align*}
A similar expression can be derived for odd $n$ and also for $s_n$, and from this one obtains \eqref{eq: ex 511}.
\subsubsection{Dimension $g=1$: CM case} Let $E$ be an elliptic curve over~$k$ with CM defined over~$k$. Its Sato--Tate group is 
$$
G:=\left\{
A_u:=\begin{pmatrix}
u & 0\\
0 & \bar u
\end{pmatrix}\,:\,|u|=1
\right\}\simeq \U(1)\,.
$$ 
The irreducible characters of $G$ are 
\begin{equation}\label{equation: charU1}
\nu_m\colon \U(1)\ra \C^\times\,,\qquad \nu_m(A_u)=u^m\qquad \text{for }m\in\Z\,.
\end{equation}
For $n\geq 1$, writing $a_1^n$ and $s_n$ in terms of the $\nu_m$, one easily finds that $S_{\tilde a_1^n}=S_{\tilde s_n}$ and that
$$
R_{\tilde a_1^n}=2\left\lceil \frac{n}{2}\right\rceil\,,\,C_{\tilde a_1^n}=\sum_{0\leq j<n/2}2 \binom{n}{j}^2\,,\quad \text{while}\quad R_{\tilde s_n}=2\,,\,C_{\tilde s_n}=2\,.
$$

\subsubsection{Dimension $g=2$: generic case}\label{section: ex g2gen}

The Sato--Tate group of an abelian surface $A$ with trivial endomorphism ring is $G:=\USp(4)$. The conjugacy classes of $G$ are in bijection with $[0,\pi]^2/\mathfrak S_2$, where $\mathfrak S_2$ denotes the symmetric group on two letters.
From \cite[Thm. 7.8.C, Chap. VII]{Weyl} one finds that the irreducible characters $\chi_{m,n}$ of $G$, for $m\geq n\geq 0$, are given by the formula
\begin{equation}\label{equation: charUSp4}
\chi_{m,n}(\alpha,\beta)=\frac{\sin((m+2)\alpha)\sin((n+1)\beta)-\sin((m+2)\beta)\sin((n+1)\alpha)}{\sin(2\alpha)\sin(\beta)-\sin(2\beta)\sin(\alpha)}\,,
\end{equation}
where $(\alpha,\beta)\in[0,\pi]^2/\mathfrak S_2$. Note that $\chi_{1,0}$ is the trace of the standard representation of $G$. From \eqref{equation: charUSp4}, one finds by direct computation that
$$
s_n= \chi_{n,0}-\chi_{n-1,1}+\chi_{n-3,1}-\chi_{n-4,0}
$$
for $n\geq 4$. Similarly one computes that $s_2=-\chi_{0,0}-\chi_{1,1}+\chi_{2,0}$ and $s_3=-\chi_{2,1}+\chi_{3,0}$, and concludes that $R_{\tilde s_n},\,C_{\tilde s_n}\leq 4$. 

Let us show on the other hand that the sequence $\{R_{\tilde a_1^n}\}_n$ (and thus also $\{C_{\tilde a_1^n}\}_n$) grows unboundedly\footnote{Using \cite[\S9, Chap. VII]{Weyl}, one can obtain a closed formula for $C_{\tilde a_1^n}$, but we will not pursue this here. We will content ourselves with listing a few of them:
$$
C_{\tilde a_1^2}=1^2+1^2+1^2\,,\quad C_{\tilde a_1^3}=3^2+2^2+ 1^2\,,\quad C_{\tilde a_1^4}=3^2+5^2+6^2+2^2+3^2+1^2\,.
$$} in $n$.
Let now $V$ denote the standard representation of $\USp(4)$. Let $W$ be the representation defined by $\Lambda^2V=W\oplus \C$, which has character $\chi_{1,1}$. By \cite[p.248]{FH04} the character $\chi_{m,n}$ is the character of a subrepresentation\footnote{Note that in \cite[\S16.2]{FH04}, it is written $\Gamma_{m-n,n}$ to denote the representation that is attached to the character that we denote by $\chi_{m,n}$.} of
\begin{equation}\label{equation: subprod}
\Sym^{m-n}(V)\otimes\Sym^{n}(W)\,.
\end{equation}
But note that the above representation is a subrepresentation of $V^{(m+n)\otimes}$, from which we deduce that $a_1^t$ contains all the irreducible characters $\chi_{m,n}$ for which $m+n=t$ and $m\geq n$.

Finally, one can use \cite[Ex. 24.20]{FH04} to show that $S_{\tilde a_1^n}$ and $S_{\tilde s_n}$ are of comparable size.

\section{Examples}\label{section: examples}
In this section, we illustrate the content of Theorem~\ref{theorem: main} by means of considering a few examples. Except of the example in~\S\ref{example: charnonoptimal} (which has a different purpose), they all follow the same pattern: we take two abelian varieties~$A$ and~$A'$ defined over $k$, of the same dimension $g\geq 1$, and similar conductors, but with distinct arithmetic invariants such as the rank or the Sato--Tate group. Let~$G$ (resp.~$G'$) denote the Sato--Tate groups of~$A$ (resp.~$A'$). We take a selfdual character~$\varphi_0$ of $\USp(2g)$ such that its restriction~$\varphi$ (resp. $\varphi'$) to~$G$ (resp.~$G'$) does not contain the trivial character. We then plot the functions $\delta(\varphi,x)^2$ and $\delta(\varphi',x)^2$ for $x$ in a wide range of values, and certify that they exhibit the behaviour predicted by Theorem~\ref{theorem: main}.

Recall that we can write
$$
I(\varphi)=I_1(\varphi)+I_2(\varphi)\,,
$$
where $I_1(\varphi)$ and $I_2(\varphi)$ are as defined in \eqref{equation: I1I2}. If $\varphi=\sum_\chi c_\chi \chi$, let us write $r_\varphi:=\sum_\chi c_\chi r_\chi$. In the examples below, it will be convenient to use the suggestive notation $r_{\varphi_0(A)}:=r_{\varphi}$. In the case that~$\varphi$ is the character of the tautological representation of~$G$, we will simply write $r_A$ for $r_\varphi$. Note that~$r_A$ is the so called \emph{analytic rank of $A$}, that is, the order of vanishing of the Hasse--Weil $L$-function of~$A$ at the central point. By definition, we have
$$
I_1(\varphi)=(2r_\varphi+u_\varphi)^2\,,
$$
where $u_\varphi$ denotes the Frobenius--Schur index of $\varphi$. 

As one can see from the proof of Proposition~\ref{proposition: boundL2}, a large conductor $N$ of $A$ with respect to the rank makes the term $I_2(\varphi)$ dominant, blurring the contribution of $I_1(\varphi)$ in the asymptotic $L^2$-norm $I(\varphi)$. In order to illustrate the influence in the convergence rate of the arithmetic invariants of~$A$ appearing in $I_1(\varphi)$ (the rank and the Sato--Tate group), we will often consider examples for which $r_\varphi$ or $r_\varphi'$ is exceptionally large with respect to the conductor $N$. This is what we call examples of `relatively large rank with respect to the conductor'. In the examples considered, it is $I_1(\varphi)$ which happens to dominate $I_2(\varphi)$. 

In \S\ref{section: odd weight} and \S\ref{section: even weight}, we will use the following lemma.

\begin{lemma}\label{lemma: FrobSchur} 
Let $\nu_m$, $\chi_n$, and $\chi_{m,n}$ denote the irreducible characters of $\U(1)$, $\USp(2)$, and $\USp(4)$ as defined in \eqref{equation: charU1}, \eqref{equation: charUSp2}, and \eqref{equation: charUSp4}. Their respective Frobenius--Schur indices are 
$$
u_{\nu_m}=
\begin{cases}
1 & \text{if }m=0\,,\\
0 & \text{otherwise,}
\end{cases}
\qquad u_{\chi_n}=(-1)^n\,, \qquad u_{\chi_{m,n}}=(-1)^{n+m}\,.
$$
\end{lemma}

\begin{proof}
The statement for $\nu_m$ is clear, provided that $\nu_m$ takes complex values if and only if $m\not=0$. By \cite[Prop. 38]{Ser77}, a character $\chi$ of a compact group $G$ has $u_\chi=1$ (resp. $u_\chi=-1$) if it corresponds to a representation $\varrho$ which is symmetric (resp. alternating)\footnote{By this we mean that the module $V$ affording $\varrho$ possesses a nonzero nondegenerate $G$-invariant bilinear form which is symmetric (resp. alternating).}. Since the standard representation of $\USp(2)$ is alternating, we deduce that $u_{\chi_n}=(-1)^n$ from the fact that the product of an alternating representation and a symmetric one is alternating, and that the product of two alternating (resp. symmetric) representations is symmetric. Let $V$ denote the standard representation of $\USp(4)$ and~$W$ be the representation of $\USp(4)$ defined in \S\ref{section: ex g2gen}, so that $\chi_{m,n}$ is the character of a subrepresentation of $\Sym^{m-n}(V)\otimes \Sym^n(W)$. Note that since $V$ is alternating, we have that~$W$ is symmetric. Since a subrepresentation of an alternating (resp. symmetric) representation is alternating (resp. symmetric), we find that $u_{\chi_{m,n}}=(-1)^{m-n}=(-1)^{m+n}$. 
\end{proof}

\begin{remark} If $\chi$ is an irreducible character of $G$, recall that we denote by $w_\chi$ its weight (as defined in Remark~\ref{remark: completedL}). By deep conjectures (see \S\ref{section: even weight}) one expects a different behaviour of $r_{\chi}$ depending on whether~$w_\chi$ is odd or even. In order to isolate these two typical behaviors, and although there is no such restriction in Theorem~\ref{theorem: main}, in all the examples below we consider virtual characters $\varphi$ such that their constituents are either all of odd weight or all of even weight.
\end{remark}

\begin{remark}\label{remark: smalljac}
In order to determine the value of the function $\delta(\varphi,x)$, we use Sutherland's library Smalljac \cite{KS08}, which computes the polynomials $L_\fp(A,T)$. One limitation of the version that we use is that it only works for abelian varieties defined over $\Q$. However, this is enough to deal with abelian varieties over an arbitrary number field $k$ which are the base change of an abelian variety defined over $\Q$ (see \S\ref{example: overK}).
\end{remark}

\subsection{Odd weight}\label{section: odd weight}

In the following examples, all the irreducible constituents of $\varphi$ and $\varphi'$ have odd weight.

\subsubsection{Example 1: The rank matters ($g=1$).}\label{example: rank}
Let $A$, $A'$ be two elliptic curves defined over $\Q$ without CM (i.e. with Sato--Tate groups $G=G'=\USp(2)$) with similar conductor and analytic ranks $r_A=1$ and $r_{A'}=2$. Let $\varphi_0=a_1$ be the character of $\USp(2)$ defined in \eqref{equation: nthmomentchar}. Using Lemma~\ref{lemma: FrobSchur}, we obtain that
$$
I_1(\varphi)=(2\cdot 1-1)^2=1\,,\qquad I_1(\varphi')=(2\cdot 2-1)^2=9\,.
$$
In Figure \ref{fig: 1} we plot the function $\delta(a_1,x)^2$ for two elliptic curves $A$ and $A'$ of similar conductor and ranks $1$ and $2$. We indeed observe a better convergence of the curve of rank $1$.
\begin{figure}
	\centering
	\includegraphics[width=4.5in]{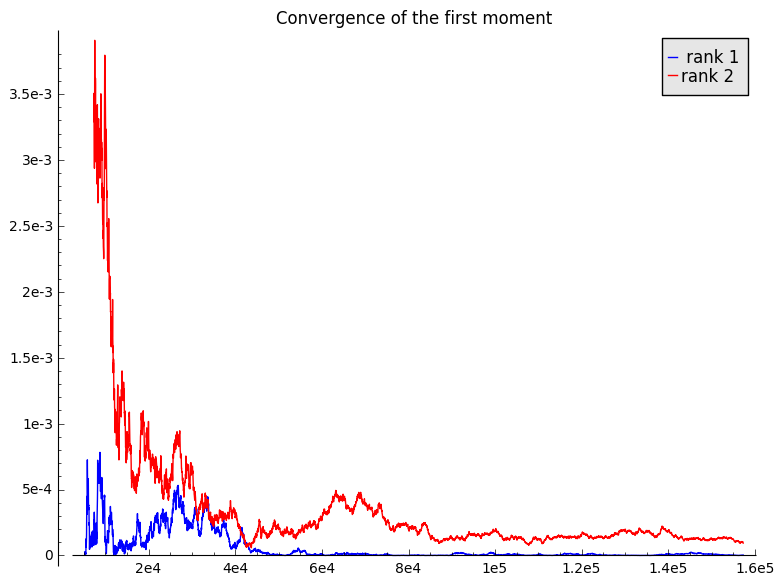}
	\caption
	{Plot of $\delta(a_1,x)^2$ for the non-CM elliptic curves with LMFDB label \href{http://www.lmfdb.org/EllipticCurve/Q/390/a/1}{390.a1} (of rank 1) and \href{http://www.lmfdb.org/EllipticCurve/Q/389/a/1}{389.a1} (of rank 2). \label{fig: 1}}
\end{figure}

\subsubsection{Example 2: The Frobenius--Schur index also matters ($g=1$).}
Let $A$, $A'$, and $\varphi_0$ be as in \S\ref{example: rank}, but suppose that now $r_A=1$ and $r_{A'}=0$. Now we have
$$
I_1(\varphi)=(2\cdot 1-1)^2=1\,,\qquad I_1(\varphi')=(2\cdot 0-1)^2=1\,.
$$
\begin{figure}
	\centering
	\includegraphics[width=4.5in]{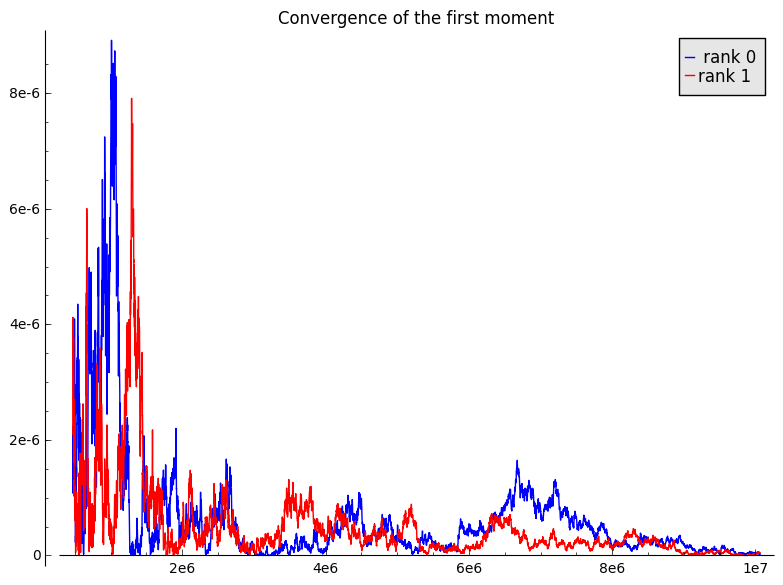}
	\caption
	{Plot of $\delta(a_1,x)^2$ for the non-CM elliptic curves with LMFDB label \href{http://www.lmfdb.org/EllipticCurve/Q/37/b/2}{37.b2} (of rank 0) and \href{http://www.lmfdb.org/EllipticCurve/Q/37/a/1}{37.a1} (of rank 1). \label{fig: 2}}
\end{figure}
In Figure \ref{fig: 2} we plot  $\delta(a_1,x)^2$ for two non-CM elliptic curves of the same conductor and ranks $0$ and $1$. In this case we observe a similar convergence rate.

\subsubsection{Example 3: The Sato--Tate group matters ($g=1$).} Let $A$ (resp. $A'$) be an elliptic curve defined over $\Q$ without (resp. with) complex multiplication, and let $\varphi_0=a_1^3$ be the cube of the trace character of $\USp(2)$. We have $G=\USp(2)$ and $G'=\U(1)\rtimes \Z/2\Z$. 

\begin{proposition}\label{proposition: a3Q}
If $\varphi_0$, $A$ and $A'$ are as above, then
$$
I_1(\varphi)=(4r_A+ 2 r_{\Sym^3(A)}-3)^2\,,\qquad I_1(\varphi')=(4r_{A'}+ 2 r_{\Sym^3(A')}-4)^2\,.
$$
\end{proposition}
\begin{proof}
From \cite[\S8.2]{Ser77}, one easily finds that the irreducible characters of $G'$ are
$$
\varrho_m:=\Tr(\Ind_{\U(1)}^{G'}(\nu_m))
$$
for $m\geq 0$, where $\nu_m$ is as in \eqref{equation: charU1}. Let $\sigma$ denote the nontrivial element of $\Z/2\Z$. For $(A_u,s)\in \U(1)\rtimes \Z/2\Z$, one has 
$$
\varrho_m( A_u, s )= 
\begin{cases}
u^m+u^{-m} & \text{if $s$ is trivial,}\\
0 & \text{if $s=\sigma$.}
\end{cases}
$$ 
It follows that $\varphi'=3\varrho_1+ \varrho_3$ and then, using Lemma~\ref{lemma: FrobSchur}, that
$$
I_1(\varphi')=(2(r_{\Sym^3(A')}-r_{A'})-1  + 3(2\cdot r_{A'}-1 ))^2 = (4 r_{A'}+2r_{\Sym^3(A')}-4)^2
$$
Provided that $\varphi=2\chi_1+\chi_3$, where $\chi_n$ is as in \eqref{equation: charUSp2}, we obtain that 
$$
I_1(\varphi)=\big((2\cdot r_{\chi_{3}}-1)+2(2\cdot r_{\chi_1}-1)\big)^2=(4r_A+ 2 r_{\Sym^3(A)}-3)^2\,,
$$
and the proposition follows.
\end{proof}
One can take now $A$ and $A'$ with similar conductors and such that $r_{\Sym^3(A)}=r_{\Sym^3(A')}$ and $r_A=r_{A'}$. If we take these analytic ranks to be respectively $1$ and $0$, we find
$$
I_1(\varphi)=1 \,,\qquad I_1(\varphi')=4\,.
$$
In Figure \ref{fig: 3} we plot $\delta(a_1^3,x)^2$ for two curves $A$, $A'$ with the above invariants. We observe that for large values of $x$, even though the two plots cross many times, the one corresponding to the CM curve seems to have a larger asymptotic $L^2$-norm. In fact, we computed that 
\begin{align*}
  I(\varphi,2^{32}) \simeq 4.08, \ \ \text{and} \ \ I(\varphi',2^{32})\simeq 7.67.
\end{align*}

\begin{figure}
	\centering
	\includegraphics[width=4.5in]{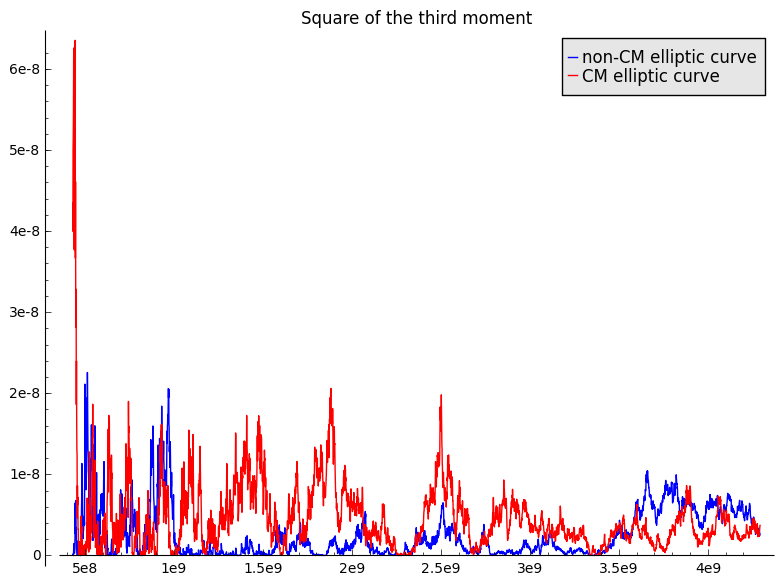}
	\caption
	{Plot of $\delta(a_1^3,x)^2$ for the rank $0$ curves \href{http://www.lmfdb.org/EllipticCurve/Q/40/a/1}{40.a1} (non-CM) and \href{http://www.lmfdb.org/EllipticCurve/Q/49/a/1}{49.a1} (CM).\label{fig: 3}}
\end{figure}

\begin{remark}
Note that, in the situation considered, if we had taken $\varphi_0=a_1$, no significant difference would have been observed between the convergence of $\delta(\varphi,x)$ and $\delta(\varphi',x)$. Indeed, in this case, one finds $I_1(\varphi)=(2\cdot r_A-1)^2=(2\cdot r_{A'}-1)^2=I_1(\varphi')$ (compare with~\S\ref{example: overK}).
\end{remark}

\subsubsection{Example 3': The Sato--Tate group matters ($g=1$).}\label{example: overK}
We will now consider an example over a finite extension of $\Q$. Let~$K$ be a quadratic imaginary field, $A$ be an elliptic curve defined over~$K$ without CM, and $A'$ an elliptic curve defined over $K$ with CM precisely by $K$. As in the previous example, let $\varphi_0=a_1^3$. Note that $\varphi'$ is a selfdual character of $\U(1)$, even if its irreducible constituents are not.

\begin{proposition} If $\varphi_0$, $A$ and $A'$ are as above, then
$$
I_1(\varphi)=(4r_A+ 2 r_{\Sym^3(A)}-3)^2\,,\qquad I_1(\varphi')=(4r_{A'}+ 2 r_{\Sym^3(A')})^2\,.
$$
\end{proposition}

\begin{proof}
Note that $\varphi'=\nu_{-3}+3\nu_{-1}+3\nu_1+\nu_3$, where $\nu_m$ is as in \eqref{equation: charU1}. Using Lemma~\ref{lemma: FrobSchur}, one finds that 
$$
I_1(\varphi')=\big((2\cdot r_{\nu_{-3}}-0)+3(2\cdot r_{\nu_{-1}}-0)+ 3(2\cdot r_{\nu_1}-0)+(2\cdot r_{\nu_3}-0)\big)^2 =(4r_{A'}+ 2 r_{\Sym^3(A')})^2\,.
$$
The computation of $I_1(\varphi)$ is exactly as in Proposition~\ref{proposition: a3Q}.
\end{proof}

One can take now $A$ and $A'$ with similar conductors and such that $r_{\Sym^3(A)}=r_{\Sym^3(A')}$ and $r_A=r_{A'}$. By Remark~\ref{remark: smalljac}, we are forced to take $A$ and $A'$ to be the base change of elliptic curves $A_0$ and $A_0'$ defined over $\Q$. Then the following lemma implies that these two analytic ranks must be even and satisfy $r_{\Sym^3(A')}\geq r_{A'}$.

\begin{lemma} Let $E$ be an elliptic curve over $\Q$ and $\tilde E$ the quadratic twist of $E$ by the quadratic extension $K/\Q$. Then
$$
r_{E_K}=r_E+r_{\tilde E}\,,\qquad r_{\Sym^3(E_K)}=r_{\Sym^3(E)}+r_{\Sym^3(\tilde E)}\,.
$$
Moreover, if $K$ is imaginary and $E$ has CM by $K$, then 
$$
r_{E_K}=2r_E\,,\qquad r_{\Sym^3(E_K)}=2r_{\Sym^3(E)},\qquad r_{\Sym^3(E_K)}\geq r_{E_K}\,.
$$ 
\end{lemma}

\begin{proof}
Let $\varepsilon$ denote the quadratic character of $K/\Q$. We have isomorphisms of $\Q_\ell[G_\Q]$-modules
$$
\Ind_{G_\Q}^{G_K}\Sym^3(V_\ell(E_K))\simeq \Sym^3(V_\ell(E))\oplus \Sym^3(V_\ell( E))\otimes \varepsilon \simeq \Sym^3(V_\ell(E))\oplus   \Sym^3(V_\ell(\tilde E))\,,
$$
and the first part of the lemma follows from the Artin Formalism of $L$-functions. For the second part, one first needs to note that, under the additional hypothesis, one has 
$$
V_\ell(E)\simeq V_\ell(E)\otimes \varepsilon \simeq V_\ell(\tilde E)\,.
$$
This follows from the fact that $\Tr(V_\ell(E))(\sigma)=0$, whenever $\varepsilon (\sigma)\not=1$ for $\sigma \in G_{\Q}$. One then concludes by noting that there is a Hecke character $\Phi$ such that 
$$\Sym^m V_\ell(E_K)\simeq \bigoplus_{j=0}^m\Phi^{m-2j}(j) \,,$$ 
where $\Phi^{m-2j}(j)$ denotes the $j$-th Tate twist of $\Phi^{m-2j}$.
\end{proof}
If we take for example $r_{\Sym^3(A)}=r_{\Sym^3(A')}=4$ and $r_A=r_{A'}= 2$, we obtain 
$$
I_1(\varphi)=144\,,\qquad I_2(\varphi')=256\,.
$$
In order to find examples of this type we have used \cite[Table 7]{MW06} for the non-CM curves, and Watkins's \texttt{Sympow} package to find the CM curves with appropriate order of vanishing of the third symmetric power. 

In Figure \ref{fig: 3 bis} we plot  $\delta(a_1^3,x)^2$ for two elliptic curves $A/K$ and $A'/K$, where $K=\Q(\sqrt{-3})$. The curve $A$ is the base change to $K$ of the curve with LMFDB label \href{http://www.lmfdb.org/EllipticCurve/Q/97448/a/2}{97448.a2}; its conductor is the ideal of $K$ generated by $97448$, it does not have CM and its ranks are $r_{A}=2$ and $r_{\Sym^3(A)}=4$. The curve $A'$ is the base change to $K$ of the CM curve with LMFDB label \href{http://www.lmfdb.org/EllipticCurve/Q/248004/g/1}{248004.g1}; its conductor is the ideal of $K$ generated by $82668$, it has CM by $K$, and its ranks are $r_{A'}=2$ and $r_{\Sym^3(A')}=4$. We observe a better convergence for the non-CM curve.
\setcounter{figure}{2}
\renewcommand{\thefigure}{\arabic{figure}'}
\begin{figure}
	\centering
	\includegraphics[width=4.5in]{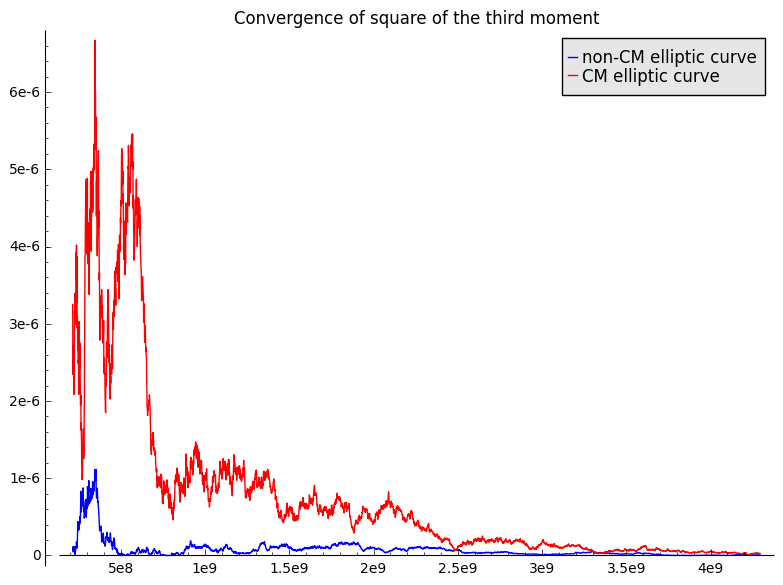}
	\caption
	{Plot of $\delta(a_1^3,x)^2$ for two elliptic curves over $\Q(\sqrt{-3})$ of rank $2$ and rank of the symmetric cube $4$, without CM and the other with CM by $K$. \label{fig: 3 bis}}
\end{figure}
\setcounter{figure}{3}
\renewcommand{\thefigure}{\arabic{figure}}

\begin{remark}
In order to plot $\delta(a_1^3,x)^2$ one needs to recover the polynomials $L_\fp(A,T)$ from the polynomials $L_p(A_0,T)$ computed by Smalljac. This is straightforward. Let $\fp$ be an unramified prime of $K$ lying over $p$. If~$\fp$ is split over~$\Q$, then Smalljac directly returns $L_\fp(A,T)=L_p(A_0,T)$. If~$\fp$ is inert over~$\Q$, then 
$$
L_\fp(A,T)=(1-\alpha_p^2)(1-\bar \alpha_p^2)=1-(a_p^2-2)T+T^2\,,
$$
where
$$
L_p(A_0,T)=(1-\alpha_p)(1-\bar \alpha_p)=1-a_pT+T^2\,.
$$
\end{remark}

\begin{remark}
For this example we could have taken $\varphi_0=a_1$. Similar computations show that then $I_1(\varphi)=4r_A^2$ and $I_1(\varphi')=(2r_A-1)^2$.
\end{remark}

\subsubsection{Example 4: The rank matters ($g=2$).} 
Let $A$, $A'$, $A''$, and $A'''$ be abelian surfaces of similar conductor, with Sato--Tate group $\USp(4)$, and respective ranks $0$, $1$, $2$, and $3$. Letting $\varphi_0=a_1$ and $\varphi$, $\varphi''$,... have the obvious meaning, we have that
$$
I_1(\varphi)=1\,,\qquad I_1(\varphi')=1\,,\qquad I_1(\varphi'')=9\,,\qquad I_1(\varphi''')=25\,.
$$  
\begin{figure}
	\centering
	\includegraphics[width=4.5in]{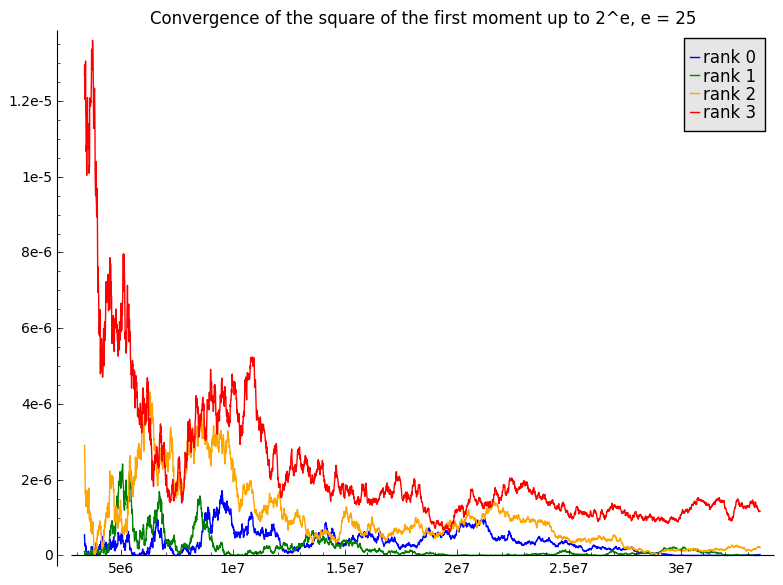}
	\caption
	{Plot of $\delta(a_1,x)^2$ for the (Jacobians of the) genus two curves with LMFDB labels \href{http://www.lmfdb.org/Genus2Curve/Q/62127/a/62127/1}{62127.a.62127.1} (of rank 0), \href{http://www.lmfdb.org/Genus2Curve/Q/61929/a/433503/1}{61929.a.433503.1} (rank 1), \href{http://www.lmfdb.org/Genus2Curve/Q/62090/a/62090/1}{62090.a.62090.1} (rank 2), and \href{http://www.lmfdb.org/Genus2Curve/Q/62411/b/62411/1}{62411.b.62411.1} (rank 3). They all have Sato--Tate group $\USp(4)$, and the number in the label indicates the conductor of the curve.\label{fig: 4}}
\end{figure}
In Figure \ref{fig: 4} we plot the function $\delta(a_1,x)^2$ for four abelian surfaces of similar conductors and ranks $0$, $1$, $2$, and $3$.

\subsubsection{Example 5: The Sato--Tate group matters ($g=2$).} Let $A$ and $A'$ be abelian surfaces defined over $\Q$ of similar conductor, with analytic rank $r_A=r_{A'}=2$, but with Sato--Tate groups $G=\USp(4)$ and $G'=\USp(2)\times \USp(2)$ (the group denoted by $G_{3,3}$ in \cite{FKRS12}). More specifically, suppose that $A'$ decomposes over $\Q$ as the product of two nonisogenous elliptic curves of rank $1$. Taking $\varphi_0=a_1$, we obtain
$$
I_1(\varphi)=(2\cdot 2-1)^2=9\,,\qquad I_1(\varphi')=(2\cdot1-1+2\cdot 1-1)^2=4\,.
$$  
An example of this phenomenon is shown in Figure \ref{fig: 5}.
\begin{figure}
	\centering
	\includegraphics[width=4.5in]{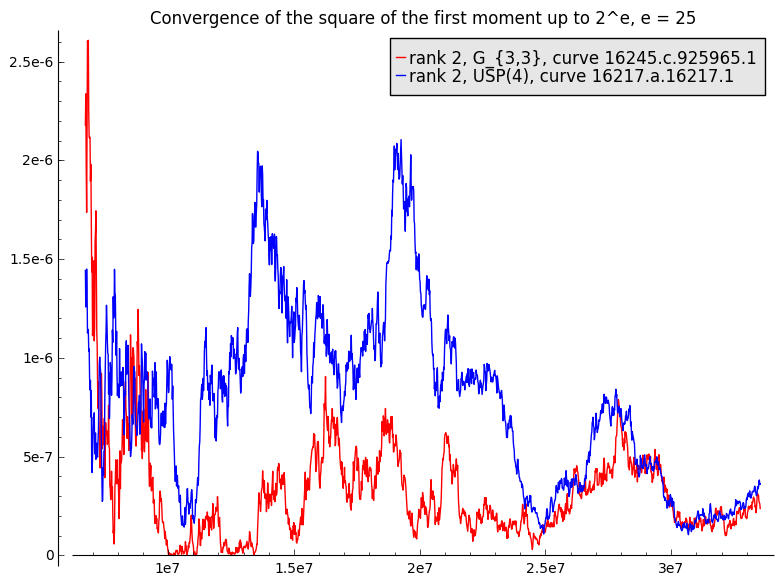}
	\caption
	{Plot of $\delta(a_1,x)^2$ for the genus two curves with LMFDB labels \href{http://www.lmfdb.org/Genus2Curve/Q/16217/a/16217/1}{16217.a.16217.1} (of rank $2$ and Sato--Tate group $\USp(4)$) and \href{http://www.lmfdb.org/Genus2Curve/Q/16245/c/925965/1}{16245.c.925965.1} (of rank $2$ and Sato--Tate group $G_{3,3}$).\label{fig: 5}}
\end{figure}
\subsubsection{Example 6: Non-optimality of the irreducible characters.}\label{example: charnonoptimal} The goal of this example is to show that, although in \S\ref{section: comments} we demonstrated that the irreducible characters of $\USp(2g)$ exhibit an extremely good convergence in the \emph{generic cases} (and, in fact, almost optimal), one can find particular examples for which their convergence is beaten by some other family of characters, which are still a basis of the central functions. For example, let $A$ be an elliptic curve without CM such that $r_A=2$ and $r_{\Sym^3(A)}=0$. Let $\chi_n$ be as in \eqref{equation: charUSp2}, and for $n\geq 0$ set 
$$
\chi_n':=
\begin{cases}
\chi_1+\chi_3 & \text{if $n=1$,}\\
\chi_n & \text{otherwise.}
\end{cases}
$$ 
Clearly, $\{\chi_n'\}_{n\geq 0}$ is also a basis of the central functions on $G$, and $I_1(\chi_n)=I_1(\chi_n')$ for every $n\not=1$. However,
$$
I_1(\chi_1)=(2\cdot 2-1)^2 > (2\cdot 2-1+2\cdot 0-1)^2=I_1(\chi_1')\,.
$$
In Figure \ref{fig: chi} we plot $\delta(\chi_1,x)^2$ and $\delta(\chi_1+\chi_3,x)^2$ for the elliptic curve $E$ with LMFDB label \href{http://www.lmfdb.org/EllipticCurve/Q/389/a/1}{389.a1}. It has $r_E=2$ and $r_{\Sym^3E}=0$. Even though $\chi_1+\chi_3$ is not always below $\chi_1$, it does seem to close a smaller area. 
\begin{figure}
	\centering
	\includegraphics[width=4.5in]{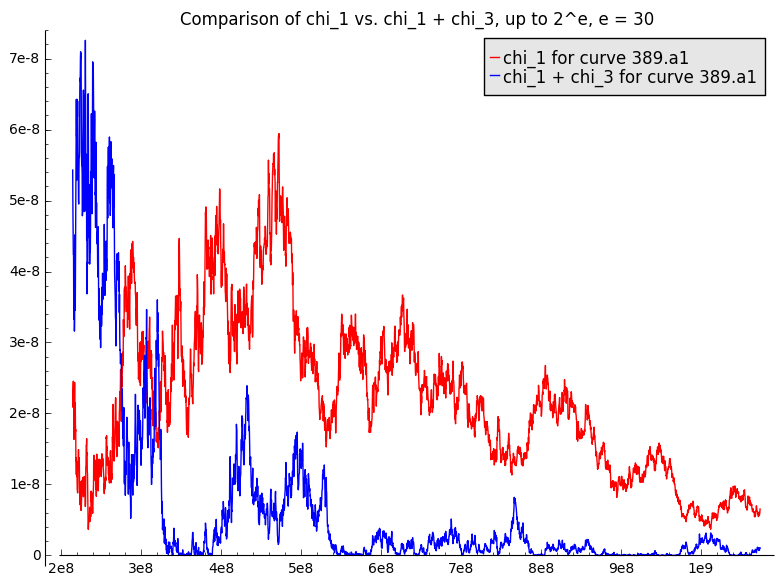}
	\caption
	{Plot of $\delta(\chi_1,x)^2$ and $\delta(\chi_1+\chi_3,x)^2$ for the elliptic curve $E$ with label \href{http://www.lmfdb.org/EllipticCurve/Q/389/a/1}{389.a1}\label{fig: chi}}
\end{figure}

\subsection{Even weight}\label{section: even weight}

Let $\chi$ be an irreducible character of $G$ and let $w_\chi$ denote its weight. If $w_\chi$ is odd, then $(w_\chi+1)/2$ is an integer, and the Bloch--Kato conjecture predicts the order of the zero at $(w_\chi+1)/2$ of the $L$-function attached to $\chi\circ \varrho_A$, which is precisely $r_\chi$. 

If $w_\chi$ is even, then $(w_\chi+1)/2$ is no longer an integer and the Bloch--Kato conjecture makes no prediction for the order of vanishing of the $L$-function attached to $\chi\circ \varrho_A$ at this point. The general philosophy is that $r_\chi$ should be $0$ unless there is a specific reason for the contrary to happen, and one thus expects $r_\chi=0$ for $\chi$ of even weight.

If $\chi_{m,n}$ is as in \eqref{equation: charUSp4}, then  \eqref{equation: subprod} implies that $\chi_{m,n}$ has even weight if and only if $m+n$ is even. One thus expects $r_{\chi_{m,n}}=0$ whenever $m+n$ is even. Using \cite[\S9, Chap. VII]{Weyl} again, one easily finds that
$$
a_2-1=\chi_{1,1}\,,\qquad a_1^2-1=\chi_{1,1}+\chi_{2,0}\,,\qquad s_2+1=-\chi_{1,1}+\chi_{2,0}\,.
$$
Let now $A$ be an abelian surface with Sato--Tate group $\USp(4)$. By Lemma~\ref{lemma: FrobSchur}, we have
$$
I_1(a_2-1)=(2\cdot 0+1)^2=1\,,\quad I_1(a_1^2-1)=(2\cdot 0+1+2\cdot 0+1)^2=4\quad I_1(s_2+1)=(-(2\cdot 0+1)+(2\cdot 0+1))^2=0\,.   
$$ 
\begin{figure}
	\centering
	\includegraphics[width=4.5in]{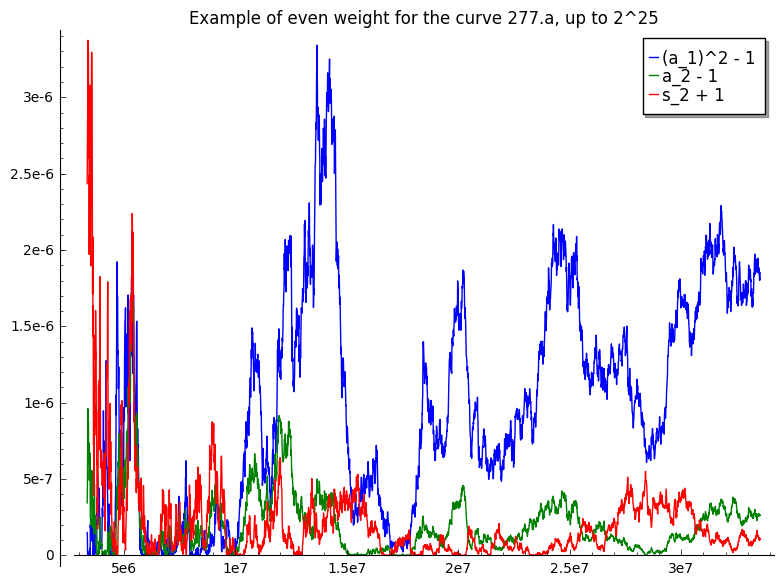}
	\caption
	{Plot of $\delta(a_1^2-1,x)^2$, $\delta(a_2-1,x)^2$, and $\delta(s_2 + 1,x)^2$ for the genus two curve \href{http://www.lmfdb.org/Genus2Curve/Q/277/a/277/1}{277.a.277.1}\label{fig: even}}
\end{figure}
In Figure \ref{fig: even} we plot $\delta(a_1^2-1,x)^2$, $\delta(a_2-1,x)^2$, and $\delta(s_2 + 1,x)^2$. Observe that, even though $I_1(a_2-1)>I_1(s_2+1)$, it is not clear from the figure whether the convergence for $s_2 + 1$ is better than the convergence for $a_2-1$. This seems to be explained by the fact that the difference $I_1(a_2-1)-I_1(s_2+1)$ is small together with the fact that $I_2(a_2-1)\leq I_2(s_2+1)$. Nonetheless, in the range of primes that we have considered, one clearly sees a slower rate of convergence for $a_1^2-1$ than for $a_2-1$ and $s_2+1$.

\section{Chebyshev bias for abelian varieties}\label{section: xevixef}
As we noted in \S\ref{section: intro}, the circle of ideas we used to study the convergence rate towards the Sato--Tate measure was introduced by Sarnak in his letter to Mazur \cite{Sar07}, in order to explain the bias that the Frobenius traces of elliptic curves have towards being positive or negative, and how the rank determines the sign of the bias. 

Not surprisingly, then, one can also use this approach to study this phenomenon for general abelian varieties. Indeed, resuming the notations of $\S\ref{section: intro}$, for a selfdual virtual character $\varphi=\sum_{\chi\not= 1} c_\chi \chi$ of $G$, recall the function
\begin{align}
  \psi(\varphi,x)=\frac{\log( x)}{\sqrt x}\sum_{|\fp|\leq x}\varphi(y_\fp).
\end{align}
Recall that by  \eqref{eq:310} we have
\begin{align*}
\lim_{X\ra \infty}  \frac{1}{\log (X)}\int_2^X {\psi(\varphi,x)}\frac{x}{dx}=  -\sum_\chi  c_\chi(2r_\chi + u_\chi)\,.
\end{align*}
In the case of elliptic curves, the bias of $a_1$ towards being positive or negative only depends on the rank. In higher dimensions, the above formula shows that it also depends on the Sato--Tate group. For example, suppose that $A$ is an abelian surface of rank  $r_A=1$ and Sato--Tate group $\USp(4)$. Then the right hand side of \eqref{eq:310} equals $-(2r_A-1)= -1$, and $a_1$ has a bias towards being negative. Suppose now that $A'$ is an abelian surface of rank $r_{A'}=1$ and Sato--Tate group $G_{3,3}$. Suppose that $A'$ is isogenous to the product of two elliptic curves, say $E$ and $E'$, of ranks $r_E=1$ and $r_{E'}=0$. Then the right hand side of \eqref{eq:310} equals $-(2r_E -1)-(2r_{E'}-1)=0$, and the $a_1$'s do not have bias towards being positive nor negative. In Figure \ref{fig: bias} we plot the function $\delta(a_1,x)$ for two abelian surfaces~$A$ and~$A'$ having these ranks and Sato--Tate groups. The prediction on the bias of the sign of $a_1$ can be clearly observed.
\begin{figure}
	\centering
	\includegraphics[width=4.5in]{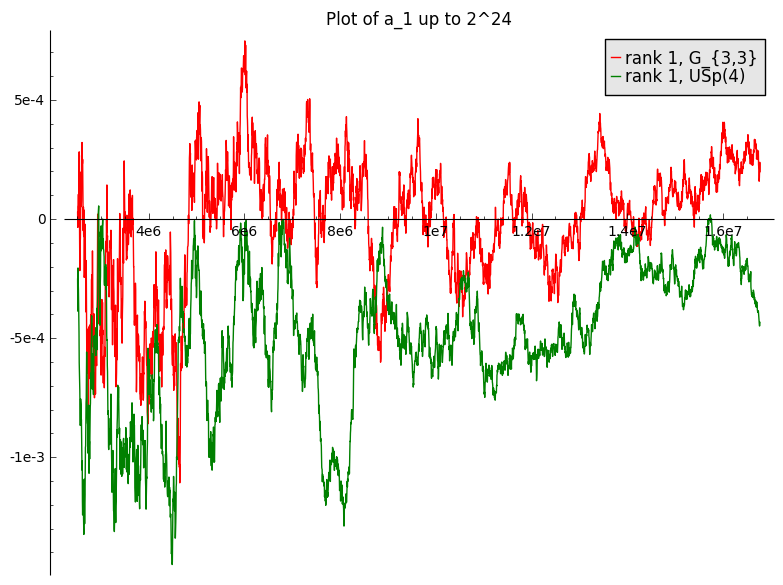}
	\caption{Plot of $\delta(a_1,x)$ for the genus two curve \href{http://www.lmfdb.org/Genus2Curve/Q/2165/a/2165/1}{2165.a.2165.1} (rank $1$ and Sato--Tate group $\USp(4)$) and for \href{http://www.lmfdb.org/Genus2Curve/Q/2156/b/34496/1}{2156.b.34496.1} (rank $1$ and Sato--Tate group $G_{3,3}$).\label{fig: bias}}
\end{figure}

\clearpage
\bibliographystyle{amsalpha}
\bibliography{refs}

\def\cprime{$'$}
\providecommand{\bysame}{\leavevmode\hbox to3em{\hrulefill}\thinspace}
\providecommand{\MR}{\relax\ifhmode\unskip\space\fi MR }
\providecommand{\MRhref}[2]{%
  \href{http://www.ams.org/mathscinet-getitem?mr=#1}{#2}
}
\providecommand{\href}[2]{#2}
\begin{thebibliography}{FKRS12}

\bibitem[ANS14]{ANS14}
Amir Akbary, Nathan Ng, and Majid Shahabi, \emph{Limiting distributions of the
  classical error terms of prime number theory}, Q. J. Math. \textbf{65}
  (2014), no.~3, 743--780. \MR{3261965}

\bibitem[BK16a]{BK16a}
Grzegorz Banaszak and Kiran~S. Kedlaya, \emph{Motivic {S}erre group, algebraic
  {S}ato-{T}ate group and {S}ato-{T}ate conjecture}, Frobenius distributions:
  {L}ang-{T}rotter and {S}ato-{T}ate conjectures, Contemp. Math., vol. 663,
  Amer. Math. Soc., Providence, RI, 2016, pp.~11--44. \MR{3502937}

\bibitem[BK16b]{BK16}
Alina Bucur and Kiran~S. Kedlaya, \emph{An application of the effective
  {S}ato-{T}ate conjecture}, Frobenius distributions: {L}ang-{T}rotter and
  {S}ato-{T}ate conjectures, Contemp. Math., vol. 663, Amer. Math. Soc.,
  Providence, RI, 2016, pp.~45--56. \MR{3502938}

\bibitem[Cog03]{Cog03a}
J.~W. Cogdell, \emph{Langlands conjectures for {${\rm GL}_n$}}, An introduction
  to the {L}anglands program ({J}erusalem, 2001), Birkh\"auser Boston, Boston,
  MA, 2003, pp.~229--249. \MR{1990381}

\bibitem[FH91]{FH04}
William Fulton and Joe Harris, \emph{Representation theory}, Graduate Texts in
  Mathematics, vol. 129, Springer-Verlag, New York, 1991, A first course,
  Readings in Mathematics. \MR{1153249}

\bibitem[Fio14]{Fiorilli}
Daniel Fiorilli, \emph{Elliptic curves of unbounded rank and {C}hebyshev's
  bias}, Int. Math. Res. Not. IMRN (2014), no.~18, 4997--5024. \MR{3264673}

\bibitem[FKRS12]{FKRS12}
Francesc Fit{\'e}, Kiran~S. Kedlaya, V{\'{\i}}ctor Rotger, and Andrew~V.
  Sutherland, \emph{Sato-{T}ate distributions and {G}alois endomorphism modules
  in genus 2}, Compos. Math. \textbf{148} (2012), no.~5, 1390--1442.
  \MR{2982436}

\bibitem[Gup87]{Gup87}
R.~K. Gupta, \emph{Characters and the {$q$}-analog of weight multiplicity}, J.
  London Math. Soc. (2) \textbf{36} (1987), no.~1, 68--76. \MR{897675}

\bibitem[IK04]{IK}
Henryk Iwaniec and Emmanuel Kowalski, \emph{Analytic number theory}, American
  Mathematical Society Colloquium Publications, vol.~53, American Mathematical
  Society, Providence, RI, 2004. \MR{2061214}

\bibitem[KS08]{KS08}
Kiran~S. Kedlaya and Andrew~V. Sutherland, \emph{Computing {$L$}-series of
  hyperelliptic curves}, Algorithmic number theory, Lecture Notes in Comput.
  Sci., vol. 5011, Springer, Berlin, 2008, pp.~312--326. \MR{2467855}

\bibitem[KS09]{KS09}
\bysame, \emph{Hyperelliptic curves, {$L$}-polynomials, and random matrices},
  Arithmetic, geometry, cryptography and coding theory, Contemp. Math., vol.
  487, Amer. Math. Soc., Providence, RI, 2009, pp.~119--162. \MR{2555991}

\bibitem[Maz08]{mazur07}
Barry Mazur, \emph{Finding meaning in error terms}, Bull. Amer. Math. Soc.
  (N.S.) \textbf{45} (2008), no.~2, 185--228. \MR{2383303}

\bibitem[MS13]{mazur-stein}
Barry Mazur and William Stein, \emph{{How Explicit is the Explicit Formula?}},
  Available at
  \url{http://www.math.harvard.edu/~mazur/papers/How.Explicit.pdf}.

\bibitem[Mur85]{Mur85}
V.~Kumar Murty, \emph{Explicit formulae and the {L}ang-{T}rotter conjecture},
  Rocky Mountain J. Math. \textbf{15} (1985), no.~2, 535--551, Number theory
  (Winnipeg, Man., 1983). \MR{823264}

\bibitem[MW06]{MW06}
Phil Martin and Mark Watkins, \emph{Symmetric powers of elliptic curve
  {$L$}-functions}, Algorithmic number theory, Lecture Notes in Comput. Sci.,
  vol. 4076, Springer, Berlin, 2006, pp.~377--392. \MR{2282937}

\bibitem[RS94]{RS}
Michael Rubinstein and Peter Sarnak, \emph{Chebyshev's bias}, Experiment. Math.
  \textbf{3} (1994), no.~3, 173--197. \MR{1329368}

\bibitem[Sar07]{Sar07}
Peter Sarnak, \emph{Letter to {B}arry {M}azur on ``{C}hebyshev's bias" for
  $\tau(p)$}.

\bibitem[Ser68]{Ser68}
Jean-Pierre Serre, \emph{Abelian {$l$}-adic representations and elliptic
  curves}, McGill University lecture notes written with the collaboration of
  Willem Kuyk and John Labute, W. A. Benjamin, Inc., New York-Amsterdam, 1968.
  \MR{0263823}

\bibitem[Ser70]{Ser69}
\bysame, \emph{Facteurs locaux des fonctions z\^eta des vari\'et\'es
  alg\'ebriques (d\'efinitions et conjectures)}, S\'eminaire
  Delange-Pisot-Poitou. Th\'eorie des nombres 11.2, 1969-1970, pp.~1--15.

\bibitem[Ser77]{Ser77}
\bysame, \emph{Linear representations of finite groups}, Springer-Verlag, New
  York-Heidelberg, 1977, Translated from the second French edition by Leonard
  L. Scott, Graduate Texts in Mathematics, Vol. 42. \MR{0450380}

\bibitem[Ser94]{Ser94}
\bysame, \emph{Propri\'et\'es conjecturales des groupes de {G}alois motiviques
  et des repr\'esentations {$l$}-adiques}, Motives ({S}eattle, {WA}, 1991),
  Proc. Sympos. Pure Math., vol.~55, Amer. Math. Soc., Providence, RI, 1994,
  pp.~377--400. \MR{1265537}

\bibitem[Shi16]{Shi16}
Yih-Dar Shieh, \emph{Character theory approach to {S}ato-{T}ate groups}, LMS J.
  Comput. Math. \textbf{19} (2016), no.~suppl. A, 301--314. \MR{3540962}

\bibitem[Wey97]{Weyl}
Hermann Weyl, \emph{The classical groups}, Princeton Landmarks in Mathematics,
  Princeton University Press, Princeton, NJ, 1997, Their invariants and
  representations, Fifteenth printing, Princeton Paperbacks. \MR{1488158}

\end{thebibliography}
\end{document}